\documentclass[a4paper,12pt,reqno]{article}
\usepackage{amsmath,amssymb,amsfonts,amsthm}
\usepackage{enumerate}
\usepackage[colorlinks,colorlinks=true,citecolor=black,menucolor=black,linkcolor=black]{hyperref}
\usepackage[pdftex]{graphicx}
\usepackage{cite}
\usepackage[top=3cm,bottom=3cm,textwidth=16cm,centering,a4paper]{geometry}
\usepackage{hyperref}
\usepackage{indentfirst}
\usepackage{type1cm}
\usepackage{color}
\usepackage{xcolor}
\definecolor{amber}{rgb}{1.0,0.75,0.0}

\newtheorem{Lemma}{Lemma}[section]
\newtheorem{Remark}[Lemma]{Remark}
\newtheorem{Theorem}[Lemma]{Theorem}

\newtheorem{Prop}[Lemma]{Proposition}

\newtheorem*{MPT}{Mountain Pass Theorem} 
\newcommand{\disp}{\displaystyle}

\makeatletter
\@addtoreset{equation}{section}
\makeatother

\title{On non-local nonlinear elliptic equations involving an eigenvalue problem}

\author{Kuan-Hsiang Wang\thanks{
E-mail address: khwang0511@gmail.com (K.-H. Wang)}, Ching-yu Chen\thanks{
E-mail address: aprilchen@nuk.edu.tw (C.-Y. Chen)}, Yueh-cheng Kuo\thanks{
E-mail address: yckuo@nuk.edu.tw (Y.-C. Kuo)}, Tsung-fang Wu\thanks{
E-mail address: tfwu@nuk.edu.tw (T.-F. Wu)}}
\date{}

\begin{document} 

\maketitle


\centerline{Department of Applied Mathematics}
\centerline{National University of Kaohsiung, Kaohsiung 811, Taiwan}

\begin{abstract} 
The existence and multiplicity of solutions for a class of non-local elliptic boundary value problems with superlinear source functions are investigated in this paper.
Using variational methods, we examine the changes arise in the solution behaviours as a result of the non-local effect.
Comparisons are made of the results here with those of the elliptic boundary value problem in the absence of the non-local term under the same prescribed conditions to highlight this effect of non-locality on the solution behaviours.
Our results here demonstrate that the complexity of the solution structures is significantly increased in the presence of the non-local effect with the possibility ranging from no permissible positive solution to three positive solutions and, contrary to those obtained in the absence of the non-local term, the solution profiles also vary depending on the superlinearity of the source functions.

\fontsize{10pt}{14pt}\selectfont
\end{abstract}

\section{Introduction}
In the present paper, we investigate the solution behaviour of a non-local elliptic boundary value problems of Kirchhoff-type with a superlinear source term, namely,
\begin{equation*}
\left\{\begin{array}{ll} 
\disp -\left(a\int_\Omega |\nabla u|^2dx+b\right)\Delta u=\lambda f(x)u+g(x)|u|^{p-2}u\ \ &\mbox{in}\ \ \Omega,\\
u=0 &\text{on}\ \partial\Omega,
\end{array}\right.
\eqno{(K_{a,\lambda})}
\end{equation*}
where $\Omega\subset\mathbb{R}^N$, $N\geq1$, is a bounded domain with smooth boundary, $2<p<2^*$ $(2^*=\frac{2N}{N-2}$ if $N\geq3$; $2^*=\infty$ if $N=1,2$) and $a, b, \lambda>0$ are real parameters. We are interested in the cases where the weight functions $f$ and $g$ are sign-changing in $\overline\Omega$ and thus we impose the following conditions:
\begin{enumerate}[$(D1)$]
\item $f\in L^\infty(\Omega)$ and $|\{x\in\Omega : f(x)>0\}|>0$,
\item $g\in L^\infty(\Omega)$ which changes sign;
\end{enumerate}
and set $b=1$ in Equation $(K_{a,\lambda})$ for simplicity.

Equation $(K_{a,\lambda})$ is a stationary variation of the generalist Kirchhoff equation,
\begin{equation}
u_{tt}-M\left(\int_\Omega |\nabla u|^2dx\right)\Delta u=f, \label{eq:K}
\end{equation}
which is an extension of the classical wave equation proposed by Kirchhoff  \cite{K} to describe the transversal oscillations of a stretched string, taking into account of the effect of changes in string length during the vibrations. The well-posedness and solvability of Equation (\ref{eq:K}) has been well investigated in general dimension and domain (see, for examples, \cite{AP, D'A-S2, L, P}), and much effort has been put into studying the solution behaviours of its stationary variants focusing on various aspects of the problem with specific formulations of $M$ and $f$, in bounded \cite{BB, D, SW1, LLS, JS} and unbounded domain \cite{LY, G, I, AF, TC, SW2, DS, DMXZ} and some have, in particular, drawn attention to the presence of the non-local effect on the solution behaviours (see, for examples, \cite{CKW, HZ, C, FIJ, DPS}). We do not intend to provide a survey on the subject and would therefore refer the interested readers to the afore mentioned references and the references therein for relevant studies.

While our focus here is also on the solution behaviour of the Kirchhoff type equation, we are particularly interested in the difference that arises with the consideration of the non-local effect.
To better explain the purpose of and the motivation behind our current study, it is necessary to reiterate some previous results on the solution structures of the semilinear boundary value problem, namely
\begin{equation}
\left\{\begin{array}{lll}
-\Delta u=\lambda f(x) u+g(x) |u|^{p-2} u,& \mbox{for}& x\in \Omega,\\
u=0,  & \mbox{for} &x\in\partial \Omega,
\end{array}\right.
\label{eq:S}
\end{equation}
by setting $a=0$ and $b=1$ in Equation ($K_{a,\lambda}$).
For the detailed analysis of this problem, we refer the readers to \cite{AT, BZ, CC}; the existence and multiplicity of solutions are summarised as follow
\begin{itemize}
\item[$(R1)$] a positive solution exists for (\ref{eq:S}) whenever $0<\lambda<\lambda_1(f)$;
\item[$(R2)$] if $\int_\Omega g\phi_1^{\,p} \,dx<0,$ there exists $\delta>0$ such that  (\ref{eq:S}) has at least two positive solutions whenever $\lambda_1(f)<\lambda<\lambda_1(f)+\delta$;
\item[$(R3)$] if $\lambda$ is sufficiently large, then no positive solution of \eqref{eq:S} is permitted;
\item [$(R4)$] if $\int_\Omega g\phi_1^{\,p} \,dx>0$, then no positive solution of \eqref{eq:S} is permitted for $\lambda>\lambda_1(f)$,
\end{itemize}
with $\lambda_1(f)$ being the positive principal eigenvalue and $\phi_1$ the corresponding positive principal eigenfunction of the problem
\begin{equation}
-\Delta u=\lambda f(x) u,\quad\mbox{for}\,\, x\in \Omega,\quad u=0,  \quad\mbox{for} \,\,x\in\partial \Omega.
\label{eq:E}
\end{equation}

\begin{Remark}\label{R1.1}
Under condition $(D1),$ there exists a sequence of eigenvalues $\{\lambda_n(f)\}$ of Equation \eqref{eq:E}
with $0<\lambda_1(f)<\lambda_2(f)\leq \cdots$ and each eigenvalue being of finite multiplicity. Denoting the principal positive eigenfunctions by $\phi_1$, we have
\begin{equation}\label{R1.2}
\lambda_1(f)=\int_\Omega|\nabla \phi_1|^2dx=\inf\left\{\int_\Omega|\nabla u|^2dx : u\in H^1_0(\Omega), \int_\Omega fu^2dx=1\right\}
\end{equation}
and
\begin{equation}\label{R1.3}
\lambda_2(f)=\inf\left\{\int_\Omega|\nabla u|^2dx : u\in H^1_0(\Omega), \int_\Omega fu^2dx=1, \int_\Omega \nabla u\nabla \phi_1dx=0\right\}.
\end{equation}
\end{Remark}

\begin{figure}[htbp]
\hspace*{0.5cm}\includegraphics[scale=0.8]{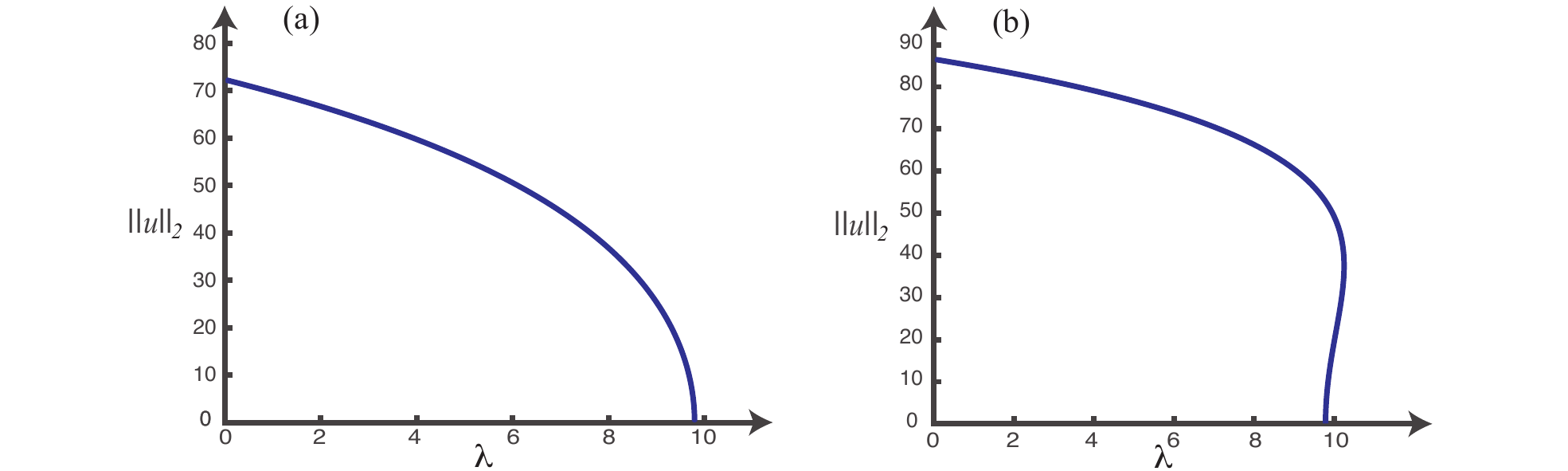}
\caption{Examples of solution structures for Equation (\ref{eq:S}) where (a) $\int_\Omega g\phi_1^{\,p} dx>0$ with $g(x)=-\sin (3\pi x)$ and (b)$\int_\Omega g\phi_1^{\,p} dx<0$ with $g(x)=-\sin (1.8\pi x)$; $p=4$ in both cases. Details of the numerical schemes are given in Section 2.}
\label{fg:a0}
\end{figure}
Results $(R1)-(R4)$ are better visualized using the two examples that we simulated numerically and presented in Figure \ref{fg:a0}.
Expressed in terms of global bifurcation theory, the results $(R1)-(R4)$ indicate that the solutions bifurcate from the branch of zero solutions in two directions depending on the sign of $\int_\Omega g\phi_1^{\,p} dx$: bifurcating to the left if $\int_\Omega g\phi_1^{\,p}dx>0$ and to the right if $\int_\Omega g\phi_1^{\,p}dx<0$, turning again to the left to give the two branches of positive solutions for $\lambda_1(f)<\lambda<\lambda_1(f)+\delta$.

The solution structures vary considerably when the non-local effect that is typical of the Kirchhoff-type equations is included into Equation (\ref{eq:S}). The summarised results of $(R1)-(R4)$ observed in the positive solutions of Equation (\ref{eq:S}) are valid for all values of $p$, whereas for the Kirchhoff-type equation we studied here, different solution behaviours are found within different regimes of $p$. Our results, which will be presented in Theorems \ref{T:p>4}-\ref{T:p<4-2} below, indicate that in the cases when $p\le 4$ Equation $(K_{a,\lambda})$ exhibits a much more complex solution structure than that of Equation ({\ref{eq:S}) with the possibility ranging from admitting no positive solution to three positive solutions when $p<4$, depending on the sign of  $\int_\Omega g\phi_1^{\,p}dx$ and varying with the values of $a$; whereas for $p=4,$ provided $a$ is sufficient large, the solution behaviours become independent on the condition imposed on the integral
$\int_\Omega g\phi_1^pdx$. When $p>4$ and in this case for all $a>0$, the non-local effect alone ensures the existence of two positive solutions without imposing any additional conditions on the integral $\int_\Omega g\phi_1^pdx$.


The sections of this paper are organised as follow. In Section 2, we present the main results in Theorems \ref{T:p>4}-\ref{T:p<4-2}, each with examples computed numerically to give graphical interpretations of these results. The proofs of these theorems are then presented in the sections that follow and we begin by showing that the energy functional is to have the mountain pass geometry while satisfying the Palais-Smale condition in Section 3. In Section 4, we prove Theorem \ref{T:p>4} and present the proofs of Theorems \ref{T:p=4+} and \ref{T:p=4-} in Section 5 and finally those of Theorems \ref{T:p<4+}-\ref{T:p<4-2} in Section 6. Furthermore, we explore the asymptotic behaviour of these positive solutions at parameter $\lambda$.

\section{Main results}
We carry out the analysis using the variational methods and the positive solutions of Equation $(K_{a,\lambda})$ are found by considering the energy functional $J_{a, \lambda} : H^1_0(\Omega) \to \mathbb{R}$ with
$$J_{a, \lambda}(u)=\frac a4\left(\int_\Omega|\nabla u|^2dx\right)^2+\frac 12\int_\Omega|\nabla u|^2dx-\frac{\lambda}{2}\int_{\Omega}fu^2dx-\frac 1p\int_{\Omega}g|u|^pdx.$$
Furthermore, $J_{a,\lambda}$ is a $C^1$ functional with the derivative given by
$$\langle J'_{a, \lambda}(u), \varphi\rangle=\left(a\int_\Omega|\nabla u|^2dx+1\right)\left(\int_\Omega\nabla u \nabla\varphi dx\right)-\lambda\int_{\Omega}fu\varphi dx-\int_{\Omega}g|u|^{p-2}u\varphi dx$$
for all $\varphi\in H^1_0(\Omega)$, where $J'_{a, \lambda}$ denotes the Fr\'echet derivative of $J_{a, \lambda}$. We conclude that $u$ is a critical point of $J_{a, \lambda}$ if and only if it satisfies Equation $(K_{a,\lambda})$.

\begin{Theorem}\label{T:p>4}
Suppose that $N=1,2,3$, $4<p<2^*$ and conditions $(D1)-(D2)$ hold. Then we have the following results.
\vspace{-6pt}
\begin{enumerate}[(i)]\itemsep -3pt
\item For each $a>0$, Equation $(K_{a,\lambda})$ has a positive solution $u^+$ with $J_{a, \lambda}(u^+)>0$ whenever $0<\lambda\leq \lambda_1(f)$.
\item For each $a>0$, there exists $\delta_a>0$ such that Equation $(K_{a,\lambda})$ has two positive solutions $u^+$ and $u^-$ with $J_{a, \lambda}(u^-)<0<J_{a, \lambda}(u^+)$ whenever $\lambda_1(f)<\lambda<\lambda_1(f)+\delta_a$. More precisely, if $\int_\Omega g\phi_1^pdx\geq0$, then
$$\delta_a=\lambda_1(f)(p-4)\left(\frac{pS_p^p}{\|g\|_\infty}\right)^{\frac{2}{p-4}}\left(\frac{a}{2p-4}\right)^{\frac{p-2}{p-4}};$$
while $\int_\Omega g\phi_1^pdx<0$, $\delta _{a}^{-}\nearrow \infty $
as $a\nearrow \infty $ and there exists a positive constant $C_0$ such that $\delta _{a}^{-}>C_{0}$ for all $a>0$.
\end{enumerate}
\end{Theorem}

\begin{figure}[htbp]
\hspace*{0.5cm}\includegraphics[scale=0.8]{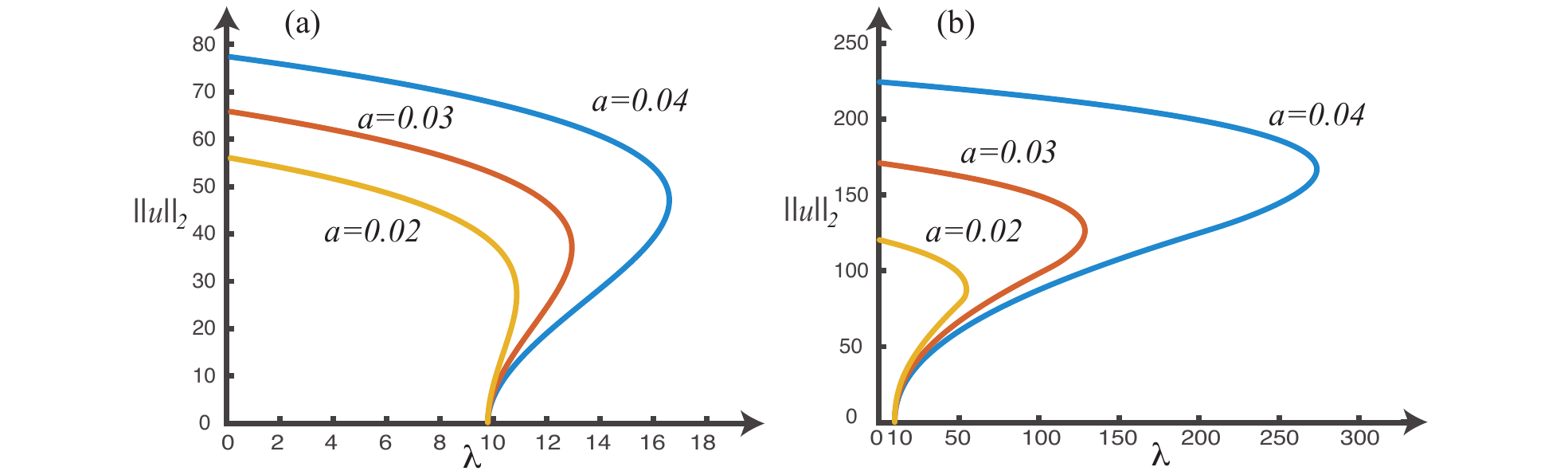}
\caption{Simulations for Theorem \ref{T:p>4} with $p=5$ showing the results of various $a$. Multiple solutions exist for both (a) $\int_\Omega g\phi_1^p dx>0$ with $g(x)=-\sin (3\pi x)$ and (b) $\int_\Omega g\phi_1^p dx<0$ with $g(x)=\sin (3\pi x)$. The bifurcation point $\lambda_1(f)$ is located at $\pi^2$ in all the simulations. }
\label{fg:p5}
\end{figure}

In this regime of $p$, the presence of the non-local effect is sufficient to guarantee the existence of two positive solutions.
We illustrate this result using the two examples in Figure \ref{fg:p5} taking $p=5$ with (a) $\int_\Omega g\phi_1^pdx>0$ on the left and (b) $\int_\Omega g\phi_1^pdx<0$ on the right.
In both examples and for each value of $a$, the existence of two branches of solutions described by Theorem \ref{T:p>4} $(ii)$ can be observed as the solution branch bifurcates to the right from the branch of zero solutions at $\lambda_1(f)$, then turning at a point  $\lambda>\lambda_1(f)+\delta_a$; this turning point being pushed further to the right with increasing value of $a$ (and hence $\delta_a$) is evident from these curves. The upper branch extending pass the initial bifurcation point $\lambda_1(f)$ towards $\lambda=0$ thus depicts the branch of positive solutions $u^+$ whenever $0<\lambda\leq \lambda_1(f)$ stated in Theorem \ref{T:p>4} $(i)$.

The difference between our results and those of Equation (\ref{eq:S}) can be observed by comparing Figure \ref{fg:p5} and Figure \ref{fg:a0}.
For demonstration purpose, we have set $p=4$ in Figure \ref{fg:a0}, the solution behaviour in general (i.e. the direction of bifurcation and the number of positive solutions) does not vary with the value of $p$ for Equation (\ref{eq:S}). In the case when $\int_\Omega g\phi_1^pdx >0$, two positive solutions are observed in Figure \ref{fg:p5} (a), whereas in Figure \ref{fg:a0} (a) only a single solution branch is present.  In Figure \ref{fg:p5} (b) and Figure \ref{fg:a0} (b) when $\int_\Omega g\phi_1^pdx <0$, both exhibit regions with two positive solutions. While this region is finite for Equation (\ref{eq:S}) as indicated by $(R3)$, in the presence of the non-local effect, this region of multiple positive solution continues to expand to the right as the non-local effect grows, leading eventually to the conclusion that at least a positive solution exists for $0<\lambda<\infty.$

Note that we have set $f\equiv1$ and $\overline\Omega=[0,1]$ in all the simulations including those of Equation (\ref{eq:S}) in Figure \ref{fg:a0}. Consequently, the positive principal eigenvalue and the corresponding eigenfunction satisfying the boundary condition are given by  $\lambda_1(f)=\pi^2$ and $\phi_1(x)=\sin(\pi x)$ respectively.
These solution branches are numerically generated using a continuation scheme and computed using matlab codes. An initial non-trivial solution is required for the scheme which is derived using a fixed point iteration method and subsequently rescaled to give a suitable starting point for the continuation method. The details of the continuation method and the fixed point scheme can be found in \cite{Keller, Kuo}.

Next, we present the cases when $p=4$ in Theorems \ref{T:p=4+} and \ref{T:p=4-} where it is necessary to consider the following:
$$\Gamma_0:=\sup_{u\in H^1_0(\Omega)\setminus\{0\}}\frac{\int_\Omega gu^4dx}{\left(\int_\Omega|\nabla u|^2dx\right)^2}.$$
By H\"older and Sobolev inequalities, we conclude that $0<\Gamma_0<\infty$. It is then possible to choose $v\in H^1_0(\Omega)$ such that $\int_\Omega gv^4dx>\int_\Omega g\phi_1^4dx$ and $\int_\Omega|\nabla v|^2dx\leq \int_\Omega|\nabla \phi_1|^2dx$ since $g$ changes sign, giving subsequently $\Gamma_0>\lambda_1(f)^{-2}\int_\Omega g\phi_1^4dx$.

\begin{Theorem}\label{T:p=4+}
Suppose that $N=1,2,3$, $p=4$ and conditions $(D1)-(D2)$ hold. If $\int_\Omega g\phi_1^4dx>0$, then we have the following results.
\vspace{-8pt}
\begin{enumerate}[(i)]\itemsep-3pt
\item For each $0<a\leq\lambda_1(f)^{-2}\int_\Omega g\phi_1^4dx$, Equation $(K_{a,\lambda})$ has a positive solution $u^+$ with $J_{a, \lambda}(u^+)>0$ whenever $0<\lambda<\lambda_1(f)$.
\item For each $\lambda_1(f)^{-2}\int_\Omega g\phi_1^4dx<a<\Gamma_0$,
\vspace{-8pt}
\begin{itemize}\itemsep-3pt
\item[(ii-1)]Equation $(K_{a,\lambda})$ has a positive solution $u^+$ with $J_{a, \lambda}(u^+)>0$ whenever $0<\lambda\leq\lambda_1(f)$$;$
\item [(ii-2)] there exists $\delta_0>0$ such that Equation $(K_{a,\lambda})$ has two positive solutions $u^+$ and $u^-$ with $J_{a, \lambda}(u^-)<0<J_{a, \lambda}(u^+)$ whenever $\lambda_1(f)<\lambda< \lambda_1(f)+\delta_0$.
\end{itemize}
\vspace{-4pt}
\item For each $a\geq\Gamma_0$, Equation $(K_{a,\lambda})$ does not admit nontrivial solution whenever $0<\lambda\leq \lambda_1(f)$.
\item For each $a>\Gamma_0$, Equation $(K_{a,\lambda})$ has a positive solution $u^-$ with $J_{a, \lambda}(u^-)<0$ whenever $\lambda>\lambda_1(f)$.
\end{enumerate}
\end{Theorem}

\begin{figure}[htbp]
\hspace*{-1.8cm}\includegraphics[scale=0.72]{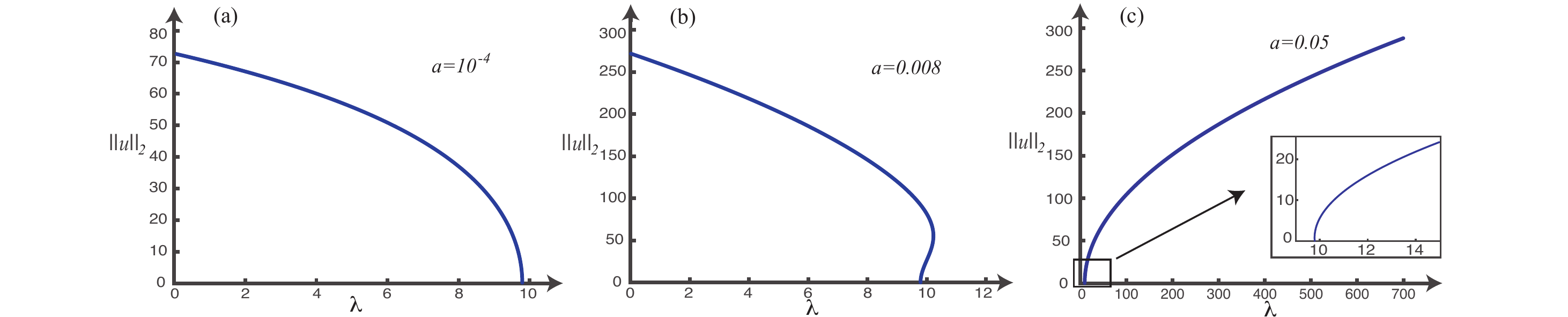}
\caption{Simulations for Theorem \ref{T:p=4+} with $p=4$ and increasing values of $a$ from (a)-(c) satisfying $\int_\Omega g\phi_1^4dx>0$; $g(x)=-\sin(3\pi x)$ is used for all three cases. }
\label{fg:p4+}
\end{figure}

Examples of Theorem \ref{T:p=4+} satisfying the conditions $\int_\Omega g\phi_1^4dx>0$ are shown in Figure \ref{fg:p4+} (a)-(c).
In (a), with $a=10^{-4}$ a single branch of positive solution describing Theorem \ref{T:p=4+} $(i)$ is observed bifurcating to the left; whereas increasing the value of $a$ to $a=0.008$ in (b),
the solution branch initially bifurcates to the right before turning at a point $\lambda>\lambda_1(f)+\delta_0$ giving two solution branches for $\lambda_1(f)<\lambda< \lambda_1(f)+\delta_0$ described by Theorem \ref{T:p=4+} ($ii$-2) and a single solution branch whenever $0<\lambda\leq\lambda_1(f)$ by Theorem \ref{T:p=4+} ($ii$-1). Figure \ref{fg:p4+} (c) demonstrates the results of Theorem \ref{T:p=4+} $(iii)$-$(iv)$ when the values of $a$ exceed a certain threshold value (namely, $\Gamma_0$). A single branch of positive solutions is seen branching off to the right without turning giving the result of case $(iv)$ that at least one positive solution exists when $\lambda>\lambda_1(f)$, and for case $(iii)$ when $0<\lambda\leq \lambda_1(f)$, no positive solution is admissible. Note that the solution behaviours of (b) and (c) are similar near the bifurcation point $\lambda_1(f)$ and thus to ensure that no turning occurs in (c), we perform the simulations until $\lambda=700$ with $\|u\|_2$ reaching the same scale as that in (b), but turning is only observed in (b).

For small values of $a$, the results here are similar to $(R1)$ and $(R4)$ of Equation (\ref{eq:S}), as demonstrated in Figure \ref{fg:p4+} (a) and Figure \ref{fg:a0} (a), both imposing the condition $\int_\Omega g\phi_1^4dx>0$. However, with a slightly larger value of $a$ in Figure \ref{fg:p4+} (b), the increased non-local effect pushes the bifurcated solution branch to the right before turning again to the left, as a result, multiple solutions appear for $\lambda>\lambda_1(f)$ immediately to the right of $\lambda_1(f)$. Increasing the value of $a$ further in Figure \ref{fg:p4+} (c), the bifurcated solution branch is prevented from turning to the left again and thus contrary to the result of $(R1)$ and $(R4)$, a positive solution exists for $\lambda>\lambda_1(f)$ while no solution is permitted for $0<\lambda\leq\lambda_1(f)$.

 \begin{Theorem}\label{T:p=4-}
Suppose that $N=1,2,3$, $p=4$ and conditions $(D1)-(D2)$ hold. If $\int_\Omega g\phi_1^4dx\leq0$, then we have the following results.
\vspace{-8pt}
\begin{enumerate}[(i)]\itemsep-3pt
\item For each $0<a<\Gamma_0$,
\vspace{-8pt}
\begin{itemize}\itemsep-3pt
\item[(i-1)] Equation $(K_{a,\lambda})$  has a positive solution $u^+$ with $J_{a, \lambda}(u^+)>0$ whenever $0<\lambda\leq\lambda_1(f)$$;$
\item[(i-2)] there exists $\delta_0>0$ such that Equation $(K_{a,\lambda})$ has two positive solutions $u^+$ and $u^-$ with $J_{a, \lambda}(u^-)<0<J_{a, \lambda}(u^+)$ whenever $\lambda_1(f)<\lambda< \lambda_1(f)+\delta_0$.
\end{itemize}
\vspace{-4pt}
\item For each $a\geq\Gamma_0$, Equation $(K_{a,\lambda})$ does not admit nontrivial solution whenever $0<\lambda\leq \lambda_1(f)$.
\item For each $a>\Gamma_0$, Equation $(K_{a,\lambda})$ has a positive solution $u^-$ with $J_{a, \lambda}(u^-)<0$ whenever $\lambda>\lambda_1(f)$.
\end{enumerate}
\end{Theorem}

\begin{figure}[htbp]
\hspace*{0.5cm}\includegraphics[scale=0.8]{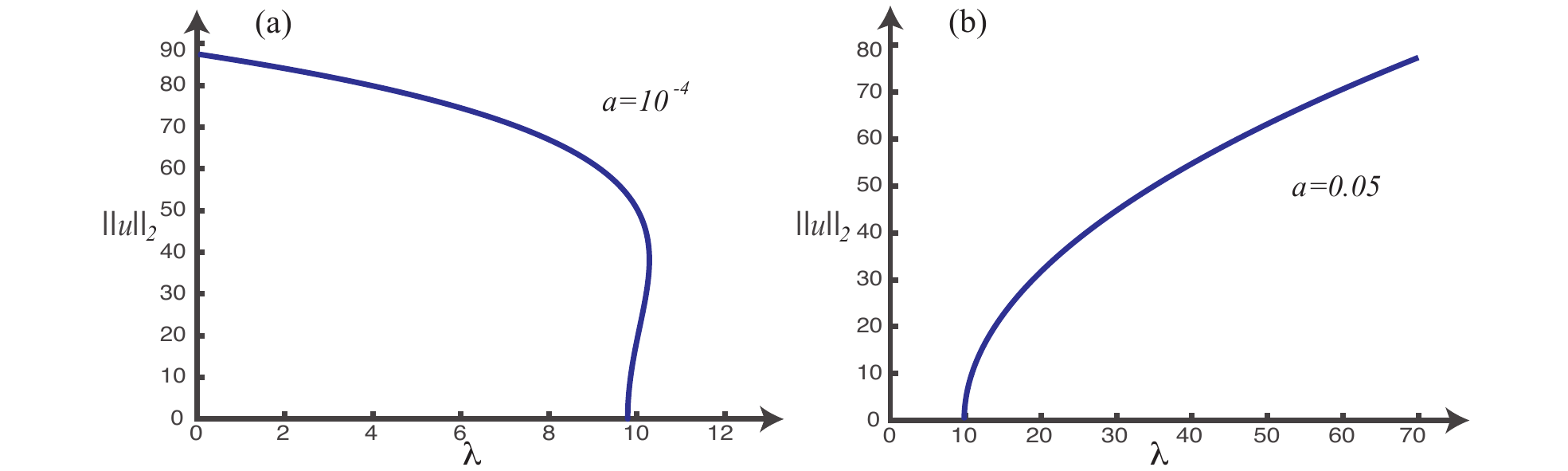}
\caption{Simulations for Theorem \ref{T:p=4-} with $p=4$ and $g(x)=-\sin(1.8\pi x)$ satisfying $\int_\Omega g\phi_1^4 dx\le0$.}
\label{fg:p4-}
\end{figure}

Figure \ref{fg:p4-} depicts the results of Theorem \ref{T:p=4-} for the cases when $\int_\Omega g\phi_1^4dx\leq0$.
In Figure \ref{fg:p4-} (a), we assume the same value of $a$ ($a=10^{-4}$) as in Figure \ref{fg:p4+} (a) to highlight the change in the solution structure as the integral $\int_\Omega g\phi_1^4dx$ changes sign. When $a$ is sufficiently small, two branches of solutions are present here with the lower branch turning left into the upper branch at a point $\lambda>\lambda_1(f)+\delta_0$
depicting the results of Theorem \ref{T:p=4-} $(i$-1) and $(i$-2). Figure \ref{fg:p4-} (b) demonstrates the results of Theorem \ref{T:p=4-} $(ii)$ and $(iii)$ when the values of $a$ exceed a certain threshold value ($\Gamma_0$) exhibiting similar behaviour to that in Figure \ref{fg:p4+} (c) when $\int_\Omega g\phi_1^4dx>0$.

The solution structure in Figure \ref{fg:p4-} (a) is similar to that in Figure \ref{fg:a0} (b) of Equation (\ref{eq:S}) when  $\int_\Omega g\phi_1^4 dx\le0$ provided $a<\Gamma_0$. Increasing the value of $a$ in Figure \ref{fg:p4-} (b) results in a profile that is also observed in Figure \ref{fg:p4+} (c) indicating that when the values of $a$ exceed a certain threshold value, namely, $a\ge\Gamma_0$, the solution behaviour becomes independent on the condition imposed on the integral $\int_\Omega g\phi_1^pdx$ for the case when $p=4$.


We next present the results for $2<p<\min\{4, 2^*\}$ when $\int_\Omega g\phi_1^pdx>0$ in Theorem \ref{T:p<4+} and for brevity, we assume the notations below that for $a>0$,
\begin{equation}\label{lam-a+}
\lambda_a^+=\lambda_1(f)\left(1-(4-p)\left(\frac{\int_\Omega g\phi_1^pdx}{p\lambda_1(f)^{\frac p2}}\right)^{\frac{2}{4-p}}\left(\frac{2p-4}{a}\right)^{\frac{p-2}{4-p}}\right)
\end{equation}
and
$$\Lambda_a^+=\lambda_1(f)\left(1-(4-p)\left(\frac{\|g\|_\infty}{2S_p^p}\right)^{\frac{2}{4-p}}\left(\frac{p-2}{a}\right)^{\frac{p-2}{4-p}}\right)$$
with $S_p$ being the best Sobolev constant for the embedding of $H^1_0(\Omega)$ into $L^p(\Omega)$.
Note that $\Lambda_a^+<\lambda_a^+<\lambda_1(f)$ for all $a>0$, and $\Lambda_a^+>0$ for all $a>(p-2)\left(4-p\right)^{\frac{4-p}{p-2}}\left(\frac12 \|g\|_\infty S_p^{-p}\right)^{\frac{2}{p-2}}$.

\begin{Theorem}\label{T:p<4+}
Suppose that $N\geq1$, $2<p<\min\{4, 2^*\}$ and conditions $(D1)-(D2)$ hold. If $\int_\Omega g\phi_1^pdx>0$, then the following results hold.
\vspace{-6pt}
\begin{enumerate}[(i)]\itemsep-3pt
\item For each $a>0$, Equation $(K_{a,\lambda})$ has two positive solutions $u^+$ and $u^-$ with $J_{a, \lambda}(u^-)<0<J_{a, \lambda}(u^+)$ whenever $\max\{0, \lambda_a^+\}<\lambda<\lambda_1(f)$.
\item For each $a>0$, Equation $(K_{a,\lambda})$ has a positive solution $u^-$ with $J_{a, \lambda}(u^-)<0$ whenever $\lambda\geq\lambda_1(f)$.
\item For each $a>(p-2)\left(4-p\right)^{\frac{4-p}{p-2}}\left(\frac12 \|g\|_\infty S_p^{-p}\right)^{\frac{2}{p-2}}$, Equation $(K_{a,\lambda})$ does not admit nontrivial solution whenever $0<\lambda<\Lambda_a^+$.
\end{enumerate}
\end{Theorem}

\begin{figure}[htbp]
\hspace*{-1.8cm}\includegraphics[scale=0.72]{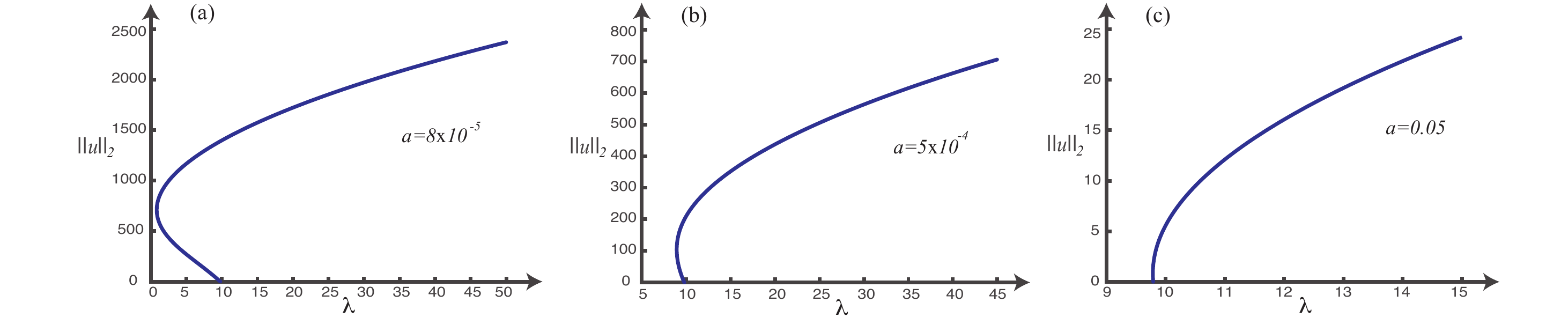}
\caption{Simulations for Theorem \ref{T:p<4+} satisfying $\int_\Omega g\phi_1^pdx>0$ with $p=3$ and $g(x)=-\sin (3\pi x)$. }
\label{fg:p3+}
\end{figure}

Examples of Theorem \ref{T:p<4+} satisfying $\int_\Omega g\phi_1^pdx>0$ in this regime of $p$ are presented in Figure \ref{fg:p3+} for three distinct values of $a$. We set $p=3$ and all three examples exhibit similar solution structure with two branches of positive solutions. The lower branch bifurcates to the left from the zero solutions at $\lambda_1(f)$ turning into the upper branch at a point $\lambda\le\max\{0, \lambda_a^+\}$ giving the results described by Theorem \ref{T:p<4+} $(i)$ and $(ii)$ that two positive solutions whenever $\max\{0, \lambda_a^+\}<\lambda<\lambda_1(f)$ and one positive solution whenever $\lambda\ge\lambda_1(f)$. Note that $\lambda_a^+\le 0$ when $0<a\ll 1$ but  approaches $\lambda_1(f)$ from the left with increasing values of $a$. Consequently, the turning point can be seen to edge closer to the bifurcation point $\lambda_1(f)$ in Figure \ref{fg:p3+} from (a) to (c) as a result of increasing $a$.

Note that the multiplicity of solutions here appears on the left hand side of the bifurcation point $\lambda_1(f)$ with the turning point approaching $\lambda=0$ for very small value of $a$ and on the right at least one positive solution is permitted for $\lambda>\lambda_1(f)$. Thus, the existence of positive solutions is always guaranteed for $\lambda>0$ provided $a$ is sufficiently small.  As $a$ increases, a region with no permissible positive solution begins to appear in a right neighbourhood of $\lambda=0$ and this neighbourhood becomes larger as the non-local effect becomes more pronounced. Nevertheless, multiple positive solutions are always present near $\lambda_1(f)$ for $\lambda<\lambda_1(f)$ unaffected by the values of $a$, however large it might be.


In the next two Theorems, we present the results for $2<p<\min\{4,2^*\}$ when $\int_\Omega g\phi_1^pdx<0$ and consider the following:
\begin{equation}\label{Ga-p}
\Gamma_p:=\sup\left\{\frac{\int_\Omega g|u|^pdx}{\left(\int_\Omega |\nabla u|^2dx\right)^{\frac p2}} : u\in H^1_0(\Omega)\setminus\{0\}, \int_\Omega fu^2dx\geq0\right\}.
\end{equation}
Under conditions $(D1)-(D2),$ we can choose a function $\varphi\in H^1_0(\Omega)$ such that $\int_\Omega f\varphi^2dx>0$ and $\int_\Omega g|\varphi|^pdx>0$ (see Proposition 6.2 in \cite{CC} for more details). Hence $\Gamma_p>0$, and it is easy to deduce that $\Gamma_p\leq\|g\|_\infty S_p^{-p}$ by Sobolev inequality. We also denote here by $a_0(p)$ a threshold value of $a$ which is dependent on the value of $p$ and is given by
\begin{equation}\label{a0}
a_0(p)=(2p-4)\left(4-p\right)^{\frac{4-p}{p-2}}\left(\frac{\Gamma_p}{p}\right)^{\frac{2}{p-2}}.
\end{equation}
In addition, for brevity, we assume the notation below that for $a>0$,
$$\Lambda_a^-=\lambda_2(f)\left(1-(4-p)\left(\frac{\|g\|_\infty}{S_p^p}\right)^{\frac{2}{4-p}}\left(\frac{2^{p-1}(p-1)\|g\|_\infty\|\phi_1\|_p^p}{\left|\int_\Omega g\phi_1dx\right|}\right)^{\frac{2p-2}{4-p}}\left(\frac{4p-8}{a}\right)^{\frac{p-2}{4-p}}\right).$$
Note that $\Lambda_a^->0$ for all
$$a>a^*(p):=(4p-8)\left(4-p\right)^{\frac{4-p}{p-2}}\left(\frac{\|g\|_\infty}{S_p^p}\right)^{\frac{2}{p-2}}\left(\frac{2^{p-1}(p-1)\|g\|_\infty\|\phi_1\|_p^p}{\left|\int_\Omega g\phi_1dx\right|}\right)^{\frac{2p-2}{p-2}}$$
and $a_0(p)<a^*(p)$.

\begin{Theorem}\label{T:p<4-1}
Suppose that $N\geq1$, $2<p<\min\{4, 2^*\}$ and conditions $(D1)-(D2)$ hold. If $\int_\Omega g\phi_1^pdx<0$, then we have the following results.
\vspace{-6pt}
\begin{enumerate}[(i)]\itemsep-3pt
\item For each $0<a<a_0(p)$,
\vspace{-6pt}
\begin{itemize}\itemsep-3pt
\item[(i-1)] Equation $(K_{a,\lambda})$ has two positive solutions $u^+$ and $u^-$ with $J_{a, \lambda}(u^-)<0<J_{a, \lambda}(u^+)$ whenever $0<\lambda\leq\lambda_1(f)$;
\item[(i-2)] there exists $\overline{\delta}_a>0$ such that Equation $(K_{a,\lambda})$ has three positive solutions $u^+$, $u^-_1$ and $u^-_2$ with $J_{a, \lambda}(u^-_1), J_{a, \lambda}(u^-_2)<0<J_{a, \lambda}(u^+)$ whenever $\lambda_1(f)<\lambda<\lambda_1(f)+\overline{\delta}_a$.
\end{itemize}
\vspace{-6pt}
\item For each $a>0$, Equation $(K_{a,\lambda})$ has a positive solution $u^-$ with $J_{a, \lambda}(u^-)<0$ whenever $\lambda>\lambda_1(f)$.
\item For each $a>a^*(p)$, Equation $(K_{a,\lambda})$ does not admit nontrivial solution whenever $0<\lambda\leq\min\{\lambda_1(f), \Lambda_a^-\}$.
\end{enumerate}
\end{Theorem}

In the case when $a\geq a_0(p)$, it is necessary to consider the following:
\begin{equation}\label{lam:a-}
\lambda_a^-:=\inf_{u\in S} \frac{\int_\Omega |\nabla u|^2dx}{\int_\Omega fu^2dx}\left(1-(4-p)\left(\frac{\int_\Omega g|u|^pdx}{p\left(\int_\Omega |\nabla u|^2dx\right)^{\frac p2}}\right)^{\frac{2}{4-p}}\left(\frac{2p-4}{a}\right)^{\frac{p-2}{4-p}}\right)\quad\text{for $a\geq a_0(p)$},
\end{equation}
where
\begin{equation}\label{S}
S=\left\{u\in H^1_0(\Omega) : \int_\Omega fu^2dx>0, \int_\Omega g|u|^pdx>0\right\}.
\end{equation}
The function $\lambda_a^-$ is nonnegative, continuous and increasing, and we conclude that
$$\lim_{a\to\infty}\lambda_a^->\lambda_1(f)$$
when $\int_\Omega g\phi_1^pdx<0$ (see Proposition \ref{P:lam:a-} in Section 3).
In particular, if $f$ is nonnegative, we can deduce that
$$\lambda_{a_0(p)}^-=0$$
(see Proposition \ref{P:lam:a2} in Section 3). Thus, when $\int_\Omega g\phi_1^pdx<0$ and $f$ is nonnegative, there is a number ${\bf A}>a_0(p)$ such that $0\leq\lambda_a^-<\lambda_1(f)$ for $a_0(p)\leq a<{\bf A}$.

\begin{Theorem}\label{T:p<4-2}
Suppose that $N\geq1$, $2<p<\min\{4, 2^*\}$, conditions $(D1)-(D2)$ hold and $\int_\Omega g\phi_1^pdx<0$. In addition, the following condition is assumed:\begin{itemize}
\item[$(D3)$] $\left|\{x\in\Omega : f(x)<0\}\right|=0$.
\end{itemize}
Then for each $a_0(p)\leq a<{\bf A}$,
\vspace{-6pt}
\begin{enumerate}[(i)]\itemsep-3pt
\item  Equation $(K_{a,\lambda})$ has two positive solutions $u^+$ and $u^-$ with $J_{a, \lambda}(u^-)<0<J_{a, \lambda}(u^+)$ whenever $\lambda_a^-<\lambda\leq\lambda_1(f)$;
\item there exists $\widehat{\delta}_a>0$ such that Equation $(K_{a,\lambda})$ has three positive solutions $u^+$, $u^-_1$ and $u^-_2$ with $J_{a, \lambda}(u^-_1), J_{a, \lambda}(u^-_2)<0<J_{a, \lambda}(u^+)$ whenever $\lambda_1(f)<\lambda<\lambda_1(f)+\widehat{\delta}_a$.
\end{enumerate}
\end{Theorem}

\begin{figure}[htbp]
\hspace*{-1cm}\includegraphics[scale=0.72]{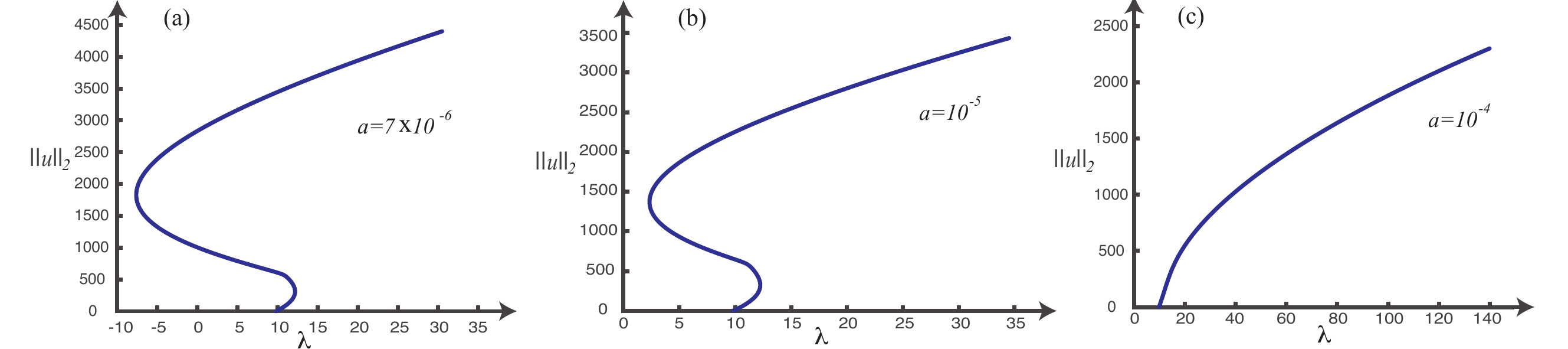}
\caption{Simulations for Theorems \ref{T:p<4-1} and \ref{T:p<4-2} satisfying  $\int_\Omega g\phi_1^pdx<0$; $p=3$ and $g(x)=\sin(3\pi x)$. }
\label{fg:p3-}
\end{figure}

The results of Theorems \ref{T:p<4-1}  and \ref{T:p<4-2} when $\int_\Omega g\phi_1^pdx<0$ are illustrated in Figure \ref{fg:p3-} where we have set $p=3\,$ as an example with the value of $a$ increases from (a) to (c). Similar solution profiles are observed in (a) and (b) with regions of $\lambda$ permitting either one, two or three positive solutions. For the small value of $a$ assumed in (a), the turning from the middle to the upper solution branch occurs in the region $\lambda<0$, thus depicting the case of Theorem  \ref{T:p<4-1} $(i$-1) for which two positive solutions are present for $0<\lambda<\lambda_1(f)$. With a slightly increased value of $a$ in (b), this turning point is shifted towards the right and now occurs in the region $0<\lambda<\lambda_a^-$ demonstrating the result of Theorem \ref{T:p<4-2} $(i)$ permitting two positive solutions whenever  $\lambda_a^-<\lambda\leq\lambda_1(f)$. As the values of $a$ increases further, this turning point continues moving to the right until it vanishes in (c) leaving a single branch of solutions. All three solutions continue their extension to the right and thus as stated in Theorem \ref{T:p<4-1} $(ii)$ at least one positive solution is guaranteed whenever $\lambda>\lambda_1(f)$. Subsequently, the region that admits no positive solutions, described by Theorem \ref{T:p<4-1} $(iii)$, lies to the left of these curves; however, when $a$ is very small, as in (a), such region does not exist for $\lambda>0.$

The presence of the non-local effect again is evident if we compare the results here with that of Figure \ref{fg:a0} (b) in which the solution branch bifurcates from $\lambda_1(f)$ to the right, turning left once towards $\lambda=0$ without turning again for the second time. Thus, contrary to the results of $(R1)$ and $(R3)$, a positive solution is always present for $\lambda>\lambda_1(f)$ when $a>0$ and a region with no permissible positive solution lies within $0<\lambda\leq\lambda_1(f)$ when $a$ becomes sufficiently large.

\section{ Palais-Smale sequence} \label{sec:P-S}
We first clarify the notations that are to be used in the analysis. Denote by $\|\cdot\|$ and $\|\cdot\|_q$ the $H^1_0(\Omega)$-norm and $L^q(\Omega)$-norm for $1\leq q\leq\infty$, respectively,
and the best Sobolev constant by $S_r$ for the embedding of $H^1_0(\Omega)$ into $L^r(\Omega)$ with $1\leq r<2^*$ and define it by
\begin{equation}\label{r}
S_r=\inf_{u\in H^1_0(\Omega)\backslash\{0\}}\frac{\|u\|}{\|u\|_r}.
\end{equation}
A strong convergence is indicated using $``\to"$ whereas the weak convergence $``\rightharpoonup"$.
The notation $o(1)$ denotes a quantity that goes to zero as $n \to \infty$.
The notation $\{u_n\}$ is used to denote both the sequence itself and its subsequence; consequently, during the analysis below the notation $\{u_n\}$ will indicate that a subsequence may be used instead if necessary without further deliberation.
The proof that is given in this section shows that the functional $J_{a,\lambda}$ possesses the mountain pass geometry and satisfies the Palais-Smale condition and we begin by recalling the well known Mountain Pass Theorem \cite{AR}.

\begin{MPT}\label{MPT}
Let $E$ be a Banach space, $J\in C^1(E, \mathbb{R}), e\in E$ and $\rho>0$ be such that $\|e\|_E>\rho$ and
$$b:=\inf_{\|u\|_E=\rho}J(u)>J(0)\geq J(e).$$
If $J$ satisfies Palais-Smale condition at level $\alpha$ with
\begin{equation*}
\alpha:=\inf_{\gamma\in \Gamma}\max_{t\in[0,1]}J(\gamma(t))\quad\mbox{and}\quad
\Gamma:=\{\gamma\in C([0,1], E) : \gamma(0)=0, \gamma(1)=e\},
\end{equation*}
then $\alpha$ is a critical value of $J$ and $\alpha\geq b$.
\end{MPT}
Thus, we say that the functional $J_{a,\lambda}$ has the mountain pass geometry if there exist $\rho>0$ and $e\in H^1_0(\Omega)$ such that
$$\|e\|>\rho\quad\text{and}\quad\inf_{\|u\|=\rho}J_{a, \lambda}(u)>J_{a,\lambda}(0) \geq J_{a,\lambda}(e),$$
and that the functional $J_{a,\lambda}$ satisfies Palais-Smale condition at level $\alpha\in\mathbb{R}$ ($(PS)_\alpha$-condition for short) if any sequence $\{u_n\}\subset H^1_0(\Omega)$ satisfying $J_{a,\lambda}(u_n)\to \alpha$ and $J_{a,\lambda}'(u_n)\to 0$ has a convergent subsequence. Such sequence is called a Palais-Smale sequence at level $\alpha$ ($(PS)_\alpha$-sequence).

To prove the mountain pass geometry of the functional $J_{a,\lambda}$ for $\lambda$ in a right neighbourhood of $\lambda_1(f)$, we decompose each $u\in H^1_0(\Omega)$ as $u=t\phi_1+w$, where $t\in\mathbb{R}$, $w\in H^1_0(\Omega)$ and $ \int_\Omega\nabla w \nabla\phi_1dx=0$.
Then
\begin{equation}\label{u1}
\|u\|^2=\lambda_1(f)t^2+\|w\|^2.
\end{equation}
Moreover, by Remark \ref{R1.1}, we have
\begin{equation}\label{u2}
\quad \lambda_2(f)\int_{\Omega}fw^2 dx\leq \|w\|^2
\end{equation}
and
\begin{equation}\label{u3}
\lambda_1(f)\int_{\Omega}f\phi_1w dx=\int_{\Omega}\nabla \phi_1\nabla w dx=0.
\end{equation}
Thus, by $\eqref{u1}-\eqref{u3}$, we deduce that
\begin{align}
\|u\|^2-\lambda\int_\Omega fu^2dx&=\lambda_1(f)t^2+\|w\|^2-\lambda\int_\Omega (t^2f\phi_1^2+2tf\phi_1w+fw^2) dx\notag\\
\ &\geq\left(1-\frac{\lambda}{\lambda_1(f)}\right)\lambda_1(f)t^2+\left(1-\frac{\lambda}{\lambda_2(f)}\right)\|w\|^2\notag\\
\ &=\left(1-\frac{\lambda}{\lambda_1(f)}\right)\|u\|^2+\lambda\left(\frac{\lambda_2(f)-\lambda_1(f)}{\lambda_1(f)\lambda_2(f)}\right)\|w\|^2 \label{u4}\\
\ &\geq\left(1-\frac{\lambda}{\lambda_1(f)}\right)\|u\|^2. \label{u5}
\end{align}
For simplicity, we use the following notations:
\begin{equation}\label{n}
\theta_1=\frac12\left(1-\frac{\lambda}{\lambda_1(f)}\right),\quad \theta_2=\frac12\left(1-\frac{\lambda}{\lambda_2(f)}\right)\quad\text{and}\quad \Theta=\theta_2-\theta_1=\frac{\lambda}{2}\left(\frac{\lambda_2(f)-\lambda_1(f)}{\lambda_1(f)\lambda_2(f)}\right).
\end{equation}

\begin{Lemma}\label{MP:p>4}
Suppose that $N=1, 2, 3$, $4<p<2^*$ and conditions $(D1)-(D2)$ hold.
Then we have the following results.
\vspace{-8pt}
\begin{enumerate}[(i)]\itemsep-4pt
\item If $\int_\Omega g\phi_1^pdx\geq0$, then for each $a>0$ there exists
$$\delta_a^+=\lambda_1(f)(p-4)\left(\frac{pS_p^p}{\|g\|_\infty}\right)^{\frac{2}{p-4}}\left(\frac{a}{2p-4}\right)^{\frac{p-2}{p-4}}$$
such that for every $0<\lambda<\lambda_1(f)+\delta_a^+$, there exist $\rho_a>0$ and $e_0\in H^1_0(\Omega)$ such that
\begin{equation}\label{p>4}
\|e_0\|>\rho_a\quad\text{and}\quad\inf_{\|u\|=\rho_a}J_{a, \lambda}(u)>0>J_{a,\lambda}(e_0).
\end{equation}
\item If $\int_{\Omega }g\phi _{1}^{p}dx<0$, then for
each $a>0$ there exist $\delta _{a}^{-}$ and a positive number $C_{0}$ with $\delta _{a}^{-}\nearrow \infty $
as $a\nearrow \infty $ and $\delta _{a}^{-}>C_{0}$ for all $a>0$ such that for every $0<\lambda <\lambda
_{1}(f)+\delta _{a}^{-}$, there exist $\rho _{a}>0$ and $e_{0}\in
H_{0}^{1}(\Omega )$ such that \eqref{p>4} holds.
\end{enumerate}
\end{Lemma}

\begin{proof}
$(i)$ By $\eqref{r}, \eqref{u5}$ and condition $(D2),$ we deduce that
\begin{align*}
J_{a, \lambda}(u)\geq\frac a4\|u\|^4+\frac 12\left(1-\frac{\lambda}{\lambda_1(f)}\right)\|u\|^2-\frac{\|g\|_\infty}{pS_p^p}\|u\|^p.
\end{align*}
Consider the function $h^+ : \mathbb{R}^+\to\mathbb{R}$ which is defined by
$$h^+(\rho)=\frac a4\rho^2+\frac12\left(1-\frac{\lambda}{\lambda_1(f)}\right)-\frac{\|g\|_\infty}{pS_p^p}\rho^{p-2}, \quad \rho>0.$$
It is easy to obtain the absolute maximum value $h^+(\rho_a)$ at
$$\rho_a=\left(\frac{apS_p^p}{(2p-4)\|g\|_\infty}\right)^{\frac{1}{p-4}}.$$
Then a direct calculation shows that $h^+(\rho_a)>0$ for every $0<\lambda<\lambda_1(f)+\delta_a^+.$
Thus, for every $0<\lambda<\lambda_1(f)+\delta_a^+$ and $\|u\|=\rho_a$, we obtain
\begin{align*}
J_{a, \lambda}(u)&\geq\rho_a^2h^+(\rho_a)>0.
\end{align*}
By condition $(D2)$, we take $\varphi\in H^1_0(\Omega)$ such that $\int_\Omega g|\varphi|^pdx>0$ and $\|\varphi\|=1$. Then for any $t>0$,
\begin{equation*}
J_{a,\lambda}(t\varphi)=\frac a4t^4+\left(\frac12-\frac{\lambda}{2}\int_\Omega f\varphi^2dx\right)t^2-\frac{\int_\Omega g|\varphi|^pdx}{p}t^p.
\end{equation*}
This implies that there exists $t_0>0$ such that $\|t_0\varphi\|>\rho_a$ and $J_{a,\lambda}(t_0\varphi)<0$.

$(ii)$ Using \eqref{u4}, we have
\begin{align}\label{J1}
 \ &\ J_{a, \lambda}(u)\notag\\
\geq&\ \frac a4\|u\|^4+\theta_1\|u\|^2+(\theta_2-\theta_1)\|w\|^2-\frac1p\int_\Omega g|t\phi_1|^pdx-\frac1p\left(\int_\Omega g|t\phi_1+w|^pdx-\int_\Omega g|t\phi_1|^pdx\right),
\end{align}
where $\theta_1$ and $\theta_2$ are as in \eqref{n}, and $u=t\phi_1+w$ with $t\in\mathbb{R}$, $w\in H^1_0(\Omega)$ and $\int_\Omega \nabla w\nabla\phi_1dx=0$.
By the mean value theorem, there exists $\theta$ with $0<\theta<1$ such that
\begin{equation}\label{g1}
\frac1p\left(\int_\Omega g|t\phi_1+w|^pdx-\int_\Omega g|t\phi_1|^pdx\right)=\int_\Omega g|t\phi_1+\theta w|^{p-2}(t\phi_1+\theta w)w dx.
\end{equation}
Young's inequality then gives
\begin{align}\label{g2}
\ &\quad \left|\int_\Omega g|t\phi_1+\theta w|^{p-2}(t\phi_1+\theta w)w dx\right|\notag\\
\ &\leq \|g\|_\infty2^{p-2}\int_\Omega\left(|t\phi_1|^{p-1}+|\theta w|^{p-1}\right)|w| dx\notag\\
\ &\leq \|g\|_\infty2^{p-2}\int_\Omega \frac{p-1}{p}B^{\frac{p}{p-1}}|t\phi_1|^p+\frac{1}{pB^p}|w|^p+|w|^p dx\notag\\
\ &=\frac{|\int_\Omega g\phi_1^pdx|}{2p}|t|^p+\frac{2^{p-2}\|g\|_\infty(1+pB^p)}{pB^pS_p^p}\|w\|^p,
\end{align}
where
\begin{equation}\label{B}
B=\left(\frac{|\int_\Omega g\phi_1^pdx|}{2^{p-1}(p-1)\|g\|_\infty\|\phi_1\|_p^p}\right)^{\frac{p-1}{p}}.
\end{equation}
Subsequently, combining $\eqref{J1}-\eqref{g2},$ we have
\begin{align}
\ & \quad J_{a,\lambda }(u)  \notag \\
\ & \geq \frac{a}{4}\Vert u\Vert ^{4}+\theta _{1}\Vert u\Vert ^{2}+(\theta
_{2}-\theta _{1})\Vert w\Vert ^{2}+\frac{|\int_{\Omega }g\phi _{1}^{p}dx|}{2p%
}|t|^{p}-\frac{2^{p-2}\Vert g\Vert _{\infty }(1+pB^{p})}{pB^{p}S_{p}^{p}}%
\Vert w\Vert ^{p}  \notag \\
\ & \geq\frac{a}{4}\Vert u\Vert ^{4}+\theta _{1}\Vert u\Vert ^{2}+(\theta
_{2}-\theta _{1})\Vert w\Vert ^{2}+\frac{|\int_{\Omega }g\phi _{1}^{p}dx|}{%
2p\lambda _{1}(f)^{\frac{p}{2}}}\left( \frac{\Vert u\Vert ^{p}}{2^{\frac p2-1}}-\Vert w\Vert
^{p}\right) -\frac{2^{p-2}\Vert g\Vert _{\infty }(1+pB^{p})}{%
pB^{p}S_{p}^{p}}\Vert w\Vert ^{p}  \notag \\
\ & = \frac{a}{4}\Vert u\Vert ^{4}+\theta _{1}\Vert u\Vert ^{2}+(\theta
_{2}-\theta _{1})\Vert w\Vert ^{2}+B_{1}\Vert u\Vert ^{p}-B_{2}\Vert w\Vert
^{p},  \label{J1'}
\end{align}
where
$$B_1=\frac{|\int_\Omega g\phi_1^pdx|}{2^{\frac p2}p\lambda_1(f)^{\frac p2}}\quad\text{and}\quad B_2=\frac{|\int_\Omega g\phi_1^pdx|}{2p\lambda_1(f)^{\frac p2}}+\frac{2^{p-2}\|g\|_\infty(1+pB^p)}{pB^pS_p^p}.$$
Let $\rho>0$ and define the function $h : [0,\rho]\to\mathbb{R}$ by
$$h(x)=\theta_1\rho^2+\frac a4\rho^4+B_1\rho^p+(\theta_2-\theta_1)x^2-B_2x^p,\quad x\in [0, \rho].$$
We conclude that there exists $\eta >0$ such that $h(x)\geq \eta $ for all $%
x\in \lbrack 0,\rho ]$  if and only if $h(0)>0$ and $h(\rho)>0$, which gives the condition $\lambda<\min\{h_1(\rho), h_2(\rho)\}$, where
\begin{equation*}
h_1(\rho)=\lambda_1(f)+\frac a2 \lambda_1(f)\rho^2+2B_1\lambda_1(f)\rho^{p-2}
\end{equation*}
and
\begin{equation*}
h_2(\rho)=\lambda_2(f)+\frac a2 \lambda_2(f)\rho^2-2(B_2-B_1)\lambda_2(f)\rho^{p-2}.
\end{equation*}
Thus, by the above argument and $\eqref{J1'},$ for any $\rho_a>0$, we conclude that for every $0<\lambda<\min\{h_1(\rho_a), h_2(\rho_a)\}$ and $\|u\|=\rho_a$,
\begin{equation}\label{J2}
J_{a,\lambda}(u)\geq\min\left\{\frac a4\rho_a^4+\theta_1\rho_a^2+B_1\rho_a^p,\ \frac a4\rho_a^4+\theta_2\rho_a^2+(B_1-B_2)\rho_a^p\right\}>0.
\end{equation}
By condition $(D2)$, we take $\varphi\in H^1_0(\Omega)$ such that $\int_\Omega g|\varphi|^pdx>0$ and $\|\varphi\|=1$. Then for any $t>0$,
\begin{equation*}
J_{a,\lambda}(t\varphi)=\frac a4t^4+\left(\frac12-\frac{\lambda}{2}\int_\Omega f\varphi^2dx\right)t^2-\frac{\int_\Omega g|\varphi|^pdx}{p}t^p.
\end{equation*}
This implies that there exists $t_0>0$ such that $\|t_0\varphi\|>\rho_a$ and $J_{a,\lambda}(t_0\varphi)<0$.

Next, we choose $\rho_a$ to give a clear range for $\lambda$.
Evidently, $h_2$ has a absolute maximum value at $\rho_0=\left(\frac{a}{(2p-4)(B_2-B_1)}\right)^{\frac{1}{p-4}}$. If $h_1(\rho_0)\geq h_2(\rho_0)$, we take $\rho_a=\rho_0$ and thus $\min\{h_1(\rho_a), h_2(\rho_a)\}=h_2(\rho_a)=\lambda_1(f)+\delta_a^-$,
where
$$\delta_a^-=\lambda_2(f)(p-4)(B_2-B_1)^{\frac{-2}{p-4}}\left(\frac{a}{2p-4}\right)^{\frac{p-2}{p-4}}+\lambda_2(f)-\lambda_1(f).$$
It is evident that $\delta_a^-$ is increasing, $\delta_a^-\to\infty$ as $a\to\infty$, and $\delta_a^->\lambda_2(f)-\lambda_1(f)$ for all $a>0$.
If $h_1(\rho_0)<h_2(\rho_0)$, then we take $\rho_a$ satisfying $h_1(\rho_a)=h_2(\rho_a)$, that is,
\begin{equation}\label{h-1}
4\left(B_1\lambda_1(f)+(B_2-B_1)\lambda_2(f)\right)\rho_a^{p-2}=a(\lambda_2(f)-\lambda_1(f))\rho_a^2+2\left(\lambda_2(f)-\lambda_1(f)\right).
\end{equation}
Thus, $\min\{h_1(\rho_a), h_2(\rho_a)\}=h_1(\rho_a)=\lambda_1(f)+\delta_a^-$, where
\begin{equation}\label{h-2}
\delta_a^-=\frac a2 \lambda_1(f)\rho_a^2+2B_1\lambda_1(f)\rho_a^{p-2}.
\end{equation}
By \eqref{h-1}, we consider the following function:
$$h^-(x)=4\left(B_1\lambda_1(f)+(B_2-B_1)\lambda_2(f)\right)x^{\frac{p-2}{2}}-a(\lambda_2(f)-\lambda_1(f))x-2\left(\lambda_2(f)-\lambda_1(f)\right),\quad x>0.$$
Let
$$x_1=\left(\frac{a(\lambda_2(f)-\lambda_1(f))}{4\left(B_1\lambda_1(f)+(B_2-B_1)\lambda_2(f)\right)}\right)^{\frac{2}{p-4}}\quad \text{and}\quad x_2=x_1+\left(\frac{\lambda_2(f)-\lambda_1(f)}{2\left(B_1\lambda_1(f)+(B_2-B_1)\lambda_2(f)\right)}\right)^{\frac{2}{p-2}}.$$
It is easy to obtain that $h^-(x_1)<h^-(\rho_a^2)=0<h^-(x_2)$, which implies that there exists a constant $D_0$ with
$$0<D_0<\left(\frac{\lambda_2(f)-\lambda_1(f)}{2\left(B_1\lambda_1(f)+(B_2-B_1)\lambda_2(f)\right)}\right)^{\frac{2}{p-2}}$$
such that
\begin{equation}\label{h-3}
\rho_a^2=\left(\frac{a(\lambda_2(f)-\lambda_1(f))}{4\left(B_1\lambda_1(f)+(B_2-B_1)\lambda_2(f)\right)}\right)^{\frac{2}{p-4}}+D_0.
\end{equation}
Combining $\eqref{h-1}-\eqref{h-3}$, we deduce that
\begin{align*}
\delta_a^-&=\frac{aB_2\lambda_1(f)\lambda_2(f)}{2\left(B_1\lambda_1(f)+(B_2-B_1)\lambda_2(f)\right)}\rho_a^2+\frac{B_1\lambda_1(f)(\lambda_2(f)-\lambda_1(f))}{B_1\lambda_1(f)+(B_2-B_1)\lambda_2(f)}\\
\ &=B_2\lambda_1(f)\lambda_2(f)\left(\frac{\lambda_2(f)-\lambda_1(f)}{2}\right)^{\frac{2}{p-4}}\left(\frac{a}{2\left(B_1\lambda_1(f)+(B_2-B_1)\lambda_2(f)\right)}\right)^{\frac{p-2}{p-4}}\\
\ &\quad +\frac{aB_2\lambda_1(f)\lambda_2(f)D_0}{2\left(B_1\lambda_1(f)+(B_2-B_1)\lambda_2(f)\right)}+\frac{B_1\lambda_1(f)(\lambda_2(f)-\lambda_1(f))}{B_1\lambda_1(f)+(B_2-B_1)\lambda_2(f)}.
\end{align*}
This completes the proof.
\end{proof}

According to Lemma \ref{MP:p>4}, when $p>4$, the functional $J_{a,\lambda}$ has the mountain pass geometry for $\lambda$ in a right neighbourhood of  $\lambda_1(f)$ whether the sign of $\int_\Omega g\phi_1^pdx$ is negative or not. Furthermore, the neighbourhood continues to expand to the right as the non-local effect grows. The mountain pass type solution is presented in Figure \ref{fg:p5} which is the upper branch of each curve.

\begin{Lemma}\label{MP:p=4}
Suppose that $N=1, 2, 3$, $p=4$ and conditions $(D1)-(D2)$ hold. Then we have the following results.
\vspace{-8pt}
\begin{enumerate}[(i)]\itemsep-4pt
\item For each $0<a<\Gamma_0$ and $0<\lambda<\lambda_1(f)$, there exist $\rho_\lambda>0$ and $e_0\in H^1_0(\Omega)$
such that
\begin{equation}\label{p=4}
\|e_0\|>\rho_\lambda\quad\text{and}\quad\inf_{\|u\|=\rho_\lambda}J_{a, \lambda}(u)>0>J_{a,\lambda}(e_0).
\end{equation}
\item If $\int_\Omega g\phi_1^4dx>0$, then for each $\lambda_1(f)^{-2}\int_\Omega g\phi_1^4dx<a<\Gamma_0$, there exists $\delta_1>0$ such that for every $\lambda_1(f)\leq\lambda<\lambda_1(f)+\delta_1$, there exist $\rho_0>0$ and $e_0\in H^1_0(\Omega)$ such that \eqref{p=4} holds.
\item If $\int_\Omega g\phi_1^4dx\leq 0$, then for each $0<a<\Gamma_0$, there exists  $\delta_1>0$ such that for every $\lambda_1(f)\leq\lambda<\lambda_1(f)+\delta_1$, there exist $\rho_0>0$ and $e_0\in H^1_0(\Omega)$ such that \eqref{p=4} holds.
\end{enumerate}
\end{Lemma}

\begin{proof}
$(i)$
By $\eqref{r}, \eqref{u5}$ and condition $(D2),$ we have
\begin{align*}
J_{a, \lambda}(u)\geq\frac a4\|u\|^4+\frac 12\left(1-\frac{\lambda}{\lambda_1(f)}\right)\|u\|^2-\frac{\|g\|_\infty}{4S_4^4}\|u\|^4.
\end{align*}
Let
\begin{equation*}
\rho_\lambda=\left(\frac14\left(1-\frac{\lambda}{\lambda_1(f)}\right)\frac{4S_4^4}{\|g\|_\infty}\right)^{\frac{1}{2}}.
\end{equation*}
Then for all $\|u\|=\rho_\lambda$, we have
\begin{equation*}
J_{a, \lambda}(u)\geq\frac a4\rho_\lambda^4+\frac14\left(1-\frac{\lambda}{\lambda_1(f)}\right)\rho_\lambda^2>0.
\end{equation*}
Since $a<\Gamma_0$, there exists $\varphi\in H^1_0(\Omega)$ such that $a\|\varphi\|^4<\int_\Omega g\varphi^4dx$. Then for any $t>0$,
\begin{equation*}
J_{a,\lambda}(t\varphi)=\left(\frac{\|\varphi\|^2}{2}-\frac{\lambda}{2}\int_\Omega f\varphi^2dx\right)t^2-\frac{\int_\Omega g\varphi^4dx-a\|\varphi\|^4}{4}t^4,
\end{equation*}
which implies that there exists $t_\lambda>0$ such that $\|t_\lambda\varphi\|>\rho_\lambda$ and $J_{a,\lambda}(t_\lambda\varphi)<0$.

$(ii)$ Using \eqref{u1} and \eqref{u4}, we have
\begin{align}
\ &\ J_{a, \lambda}(u)\notag\\
\geq&\ \frac a4(\lambda_1(f)t^2+\|w\|^2)^2+\theta_1\|u\|^2+\Theta\|w\|^2-\frac{\int_\Omega g\phi_1^4dx}{4}t^4 -\frac14\left(\int_\Omega g(|t\phi_1+w|^4-|t\phi_1|^4)dx\right)\notag\\
\geq&\ \frac a4\|w\|^4-|\theta_1|\|u\|^2+\Theta\|w\|^2+\frac{a\lambda_1(f)^2-\int_\Omega g\phi_1^4dx}{4}t^4-\frac14\left(\int_\Omega g(|t\phi_1+w|^4-|t\phi_1|^4)dx\right), \label{J41}
\end{align}
where $\theta_1$ and $\Theta$ are as in \eqref{n}, and $u=t\phi_1+w$ with $t\in\mathbb{R}$, $w\in H^1_0(\Omega)$ and $\int_\Omega \nabla w\nabla\phi_1dx=0$.
Denote
$$\Phi_4(\phi_1)=a\lambda_1(f)^2-\int_\Omega g\phi_1^4dx.$$
We have $\Phi_4(\phi_1)>0$ since $a>\lambda_1(f)^{-2}\int_\Omega g\phi_1^4dx$.
Repeating the same process in \eqref{g1} and \eqref{g2} then gives
\begin{equation}\label{g41}
\frac14\left(\int_\Omega g(|t\phi_1+w|^4-|t\phi_1|^4)dx\right)\leq\left(\frac{\Phi_4(\phi_1)}{8}t^4+\frac{\|g\|_\infty(1+4B_4^4)}{B_4^4S_4^4}\|w\|^4\right),
\end{equation}
where $B_4=\left(\frac{\Phi_4(\phi_1)}{24\|g\|_\infty\|\phi_1\|_4^4}\right)^{\frac{3}{4}}$.
Subsequently, combining \eqref{J41} and \eqref{g41}, we deduce that
\begin{align}
\ &\quad J_{a, \lambda}(u)\notag\\
\ &\geq\frac a4\|w\|^4-|\theta_1|\|u\|^2+\Theta\|w\|^2+\frac{\Phi_4(\phi_1)}{8}t^4-\frac{\|g\|_\infty(1+4B_4^4)}{B_4^4S_4^4}\|w\|^4\notag\\
\ &=\frac a4\|w\|^4-|\theta_1|\|u\|^2+\Theta\|w\|^2+\frac{\Phi_4(\phi_1)}{8\lambda_1(f)^2}(\|u\|^2-\|w\|^2)^2-\frac{\|g\|_\infty(1+4B_4^4)}{B_4^4S_4^4}\|w\|^4\notag\\
\ &\geq \frac a4\|w\|^4-|\theta_1|\|u\|^2+\Theta\|w\|^2+\frac{\Phi_4(\phi_1)}{8\lambda_1(f)^2}\left(\frac{\|u\|^4}{2}-\|w\|^4\right)-\frac{\|g\|_\infty(1+4B_4^4)}{B_4^4S_4^4}\|w\|^4\notag\\
\ &=-|\theta_1|\|u\|^2+\frac{\Phi_4(\phi_1)}{16\lambda_1(f)^2}\|u\|^4+\|w\|^2\left(\Theta-\left(\frac{\Phi_4(\phi_1)}{8\lambda_1(f)^2}+\frac{\|g\|_\infty(1+4B_4^4)}{B_4^4S_4^4}\right)\|w\|^2\right)+\frac a4\|w\|^4. \label{J42}
\end{align}
Since $\lambda\geq\lambda_1(f)$, we have
$$\Theta\geq\frac{\lambda_2(f)-\lambda_1(f)}{2\lambda_2(f)}>0.$$
Let
\begin{equation}\label{rho4}
\rho_0=\left(\frac{\lambda_2(f)-\lambda_1(f)}{2\lambda_2(f)}\right)^{\frac12}\left(\frac{\Phi_4(\phi_1)}{8\lambda_1(f)^2}+\frac{\|g\|_\infty(1+4B_4^4)}{B_4^4S_4^4}\right)^{-\frac12}.
\end{equation}
Then
\begin{equation}\label{J43}
\Theta-\left(\frac{\Phi_4(\phi_1)}{8\lambda_1(f)^2}+\frac{\|g\|_\infty(1+4B_4^4)}{B_4^4S_4^4}\right)\|w\|^2\geq0\quad\text{ for every $0\leq\|w\|\leq\rho_0$}.
\end{equation}
Moreover, there exists $\delta_1>0$ such that
\begin{equation}\label{J44}
|\theta_1|=\left|\frac12\left(1-\frac{\lambda}{\lambda_1(f)}\right)\right|\leq\frac{\Phi_4(\phi_1)}{32\lambda_1(f)^2}\rho_0^2 \quad\text{ for every } \lambda_1(f)\leq\lambda<\lambda_1(f)+\delta_1.
\end{equation}
It follows from $\eqref{J42}-\eqref{J44}$ that for every $\lambda_1(f)\leq\lambda<\lambda_1(f)+\delta_1$ and $\|u\|=\rho_0$,
$$J_{a,\lambda}(u)\geq\frac{\Phi_4(\phi_1)}{32\lambda_1(f)^2}\rho_0^4>0.$$
Since $a<\Gamma_0$, there exists $\varphi\in H^1_0(\Omega)$ such that $a\|\varphi\|^4<\int_\Omega g\varphi^4dx$. Then for any $t>0$,
\begin{equation*}
J_{a,\lambda}(t\varphi)=\left(\frac{\|\varphi\|^2}{2}-\frac{\lambda}{2}\int_\Omega f\varphi^2dx\right)t^2-\frac{\int_\Omega g\varphi^4dx-a\|\varphi\|^4}{4}t^4.
\end{equation*}
This implies that there exists $t_0>0$ such that $\|t_0\varphi\|>\rho_0$ and $J_{a,\lambda}(t_0\varphi)<0$.

$(iii)$ The proof is identical to that in part $(ii)$ and is omitted here.
\end{proof}

\begin{Lemma}\label{MP:p<4+}
Suppose that $N\geq1$, $2<p<\min\{4, 2^*\}$ and conditions $(D1)-(D2)$ hold. If $\int_\Omega g\phi_1^pdx>0$, then for each $a>0$, and $\max\{0, \lambda_a^+\}<\lambda<\lambda_1(f)$, there exist $\rho_{a,\lambda}^+>0$ and $e_0\in H^1_0(\Omega)$ such that
\begin{equation}\label{*3}
\|e_0\|>\rho_{a, \lambda}^+\quad\text{and}\quad\inf_{\|u\|=\rho_{a,\lambda}^+}J_{a,\lambda}(u)>0>J_{a, \lambda}(e_0),
\end{equation}
where $\lambda_a^+$ is as in \eqref{lam-a+}.
\end{Lemma}

\begin{proof}
By $\eqref{r}, \eqref{u5}$ and condition $(D2),$ we deduce that for $\max \left\{ 0,\lambda _{a}^{+}\right\}<\lambda<\lambda_1(f)$ and $\|u\|=\rho_{a,\lambda}^+$
\begin{align*}
J_{a, \lambda}(u)&\geq\frac a4\|u\|^4+\frac12\left(1-\frac{\lambda}{\lambda_1(f)}\right)\|u\|^2-\frac{\|g\|_\infty}{pS_p^p}\|u\|^p\\
\ &\geq \frac a4\rho_{a,\lambda}^4+\frac14\left(1-\frac{\lambda}{\lambda_1(f)}\right)\rho_{a,\lambda}^2>0,
\end{align*}
where
\begin{equation*}
\rho _{a,\lambda }^+:=\min \left\{ \left[ \frac{1}{4}\left( 1-\frac{\lambda }{%
\lambda _{1}(f)}\right) \frac{pS_{p}^{p}}{\Vert g\Vert _{\infty }}\right] ^{%
\frac{1}{p-2}},\left( \frac{(p-2)\int_{\Omega }g\phi _{1}^{p}dx}{2a\lambda
_{1}(f)^{2}}\right) ^{\frac{1}{4-p}}\Vert \phi _{1}\Vert \right\} .
\end{equation*}%
Moreover, we take
\begin{equation}\label{t0+}
t_{0}:=\left( \frac{(2p-4)\int_{\Omega }g\phi _{1}^{p}dx}{ap\lambda
_{1}(f)^{2}}\right) ^{\frac{1}{4-p}}.
\end{equation}
Then $\|t_0\phi_1\|>\rho_{a,\lambda}^+$ and we deduce that
\begin{align*}
J_{a,\lambda }(t_{0}\phi _{1})& =\frac{t_{0}^{2}}{2}\left( \lambda
_{1}(f)-\lambda \right) +t_{0}^{p}\left( \frac{a}{4}\lambda
_{1}(f)^{2}t_{0}^{4-p}-\frac{\int_{\Omega }g\phi _{1}^{p}dx}{p}\right)  \\
\ & =\frac{t_{0}^{2}}{2}\left( \lambda _{1}(f)-\lambda -\frac{%
(4-p)\int_{\Omega }g\phi _{1}^{p}dx}{p}\left( \frac{(2p-4)\int_{\Omega
}g\phi _{1}^{p}dx}{ap\lambda _{1}(f)^{2}}\right) ^{\frac{p-2}{4-p}}\right)
\\
& =\frac{t_{0}^{2}}{2}\left( \lambda _{a}^{+}-\lambda \right) <0\text{ for
all }\lambda >\max \left\{ 0,\lambda _{a}^{+}\right\} .
\end{align*}%
This completes the proof.
\end{proof}

According to Lemma \ref{MP:p<4+}, the functional $J_{a,\lambda}$ always has the mountain pass geometry in a left neighbourhood of $\lambda_1(f)$ for all $a>0$ when $\int_\Omega g\phi_1^pdx>0$.
The mountain pass type solution is presented in Figure \ref{fg:p3+} which is the lower branch of each curve.

Next, we study the mountain pass geometry of $J_{a,\lambda}$ for $2<p<\min\{4, 2^*\}$ when $\int_\Omega g\phi_1^pdx<0$. We recall the following which is defined in \eqref{Ga-p}:
\begin{equation*}
\Gamma_p:=\sup\left\{\frac{\int_\Omega g|u|^pdx}{\|u\|^p} : u\in H^1_0(\Omega)\setminus\{0\}, \int_\Omega fu^2dx\geq0\right\}. 
\end{equation*}
Then we have the following results.
\begin{Prop}\label{P:Ga-p}
Suppose that $N\geq1$, $2<p<\min\{4, 2^*\}$ and conditions $(D1)-(D2)$ hold. Then $0<\Gamma_p\leq \|g\|_\infty S_p^{-p}$, and $\Gamma_p$ is attained.
\end{Prop}

\begin{proof}
Under conditions $(D1)-(D2)$, we can choose a function $\varphi\in H^1_0(\Omega)$ such that $\int_\Omega f\varphi^2dx>0$ and $\int_\Omega g|\varphi|^pdx>0$ (see Proposition 6.2 in \cite{CC} for more details). Hence $\Gamma_p>0$. Then by Sobolev inequality, we deduce that $\Gamma_p\leq \|g\|_\infty S_p^{-p}$.

It is clear that
$$\Gamma_p=\sup\left\{\int_\Omega g|u|^pdx : u\in H^1_0(\Omega), \|u\|=1, \int_\Omega fu^2dx\geq0\right\}.$$
Thus, let $\{u_n\}$ be a maximizing sequence for $\Gamma_p$ with $\|u_n\|=1$ for all $n$. Then there exist a subsequence $\{u_n\}$ and $u_0\in H^1_0(\Omega)$ such that $u_n \rightharpoonup u_0$ in $H^1_0(\Omega)$ and
\begin{equation}\label{Ga-p:1}
u_n\to u_0\ \text{in $L^r(\Omega)$ for all $1\leq r<2^*$.}
\end{equation}
By conditions $(D1)-(D2)$ and \eqref{Ga-p:1}, we have
$$\int_\Omega fu_0^2dx=\lim_{n\to\infty}\int_\Omega fu_n^2dx\geq0,$$
and
\begin{equation}\label{Ga1}
\int_\Omega g|u_0|^pdx=\lim_{n\to\infty}\int_\Omega g|u_n|^pdx=\Gamma_p.
\end{equation}
By \eqref{Ga1} and $\Gamma_p>0$, we conclude that $u_0\neq0$.
We next show that $u_n\to u_0$ in $H^1_0(\Omega)$. If not, we have $\|u_0\|<\liminf_{n\to\infty}\|u_n\|=1$.
Let $v_0=\frac{u_0}{\|u_0\|}$. Then $\|v_0\|=1$ and $\int_\Omega fv_0^2dx\geq0$. However,
$$\int_\Omega g|v_0|^pdx=\frac{\int_\Omega g|u_0|^pdx}{\|u_0\|^p}>\int_\Omega g|u_0|^pdx=\Gamma_p$$
gives a contradiction. Hence $u_n\to u_0$ in $H^1_0(\Omega)$, and $\Gamma_p$ is attained.
\end{proof}

For $2<p<\min\{4,2^*\}$, let $B>0$ be as in \eqref{B}.
For $a,\lambda >0,$ let%
\begin{equation}
\overline{\delta }_{a}=\frac{\left\vert \int_{\Omega }g\phi
_{1}^{p}dx\right\vert }{2^{\frac{p}{2}}p\lambda _{1}(f)^{\frac{p-2}{2}}}\min
\left\{ \rho _{0}^{p-2},\rho _{1, a}^{p-2}\right\}   \label{J20}
\end{equation}%
and
\begin{equation}
\overline{\rho }_{a,\lambda }=\left\{
\begin{array}{ll}
\min \left\{\rho _{\lambda }, \rho _{1, a}\right\} , & \text{ for }%
0<\lambda <\lambda _{1}\left( f\right) , \\
\min \left\{ \rho _{0},\rho _{1,a}\right\} , & \text{ for }\lambda
_{1}\left( f\right) \leq \lambda <\lambda _{1}\left( f\right) +\overline\delta _{a}%
\end{array}%
\right. ,  \label{J21}
\end{equation}%
where
\begin{eqnarray}
\rho _{0} &=&\left( \frac{\lambda _{2}(f)-\lambda _{1}(f)}{2\lambda _{2}(f)}%
\right) ^{\frac{1}{p-2}}\left( \frac{\left\vert \int_{\Omega }g\phi
_{1}^{p}dx\right\vert }{2p\lambda _{1}(f)^{\frac{p}{2}}}+\frac{2^{p-2}\Vert
g\Vert _{\infty }(1+pB^{p})}{pB^{p}S_{p}^{p}}\right) ^{\frac{-1}{p-2}}, \label{rho-0} \\
\rho _{\lambda} &=&\left( \frac14\left(1-\frac{\lambda}{\lambda_1(f)}\right)\frac{pS_{p}^{p}}{\Vert g\Vert _{\infty }}\right) ^{\frac{1}{p-2}}%
\text{ for }0<\lambda <\lambda _{1}\left( f\right) \label{rho-lam},\\
\rho _{1, a} &=&\left( \frac{(p-2)\Gamma _{p}}{ap}\right) ^{\frac{1}{4-p}}.
\end{eqnarray}%
Then we have the following result.

\begin{Lemma}\label{MP:p<4-1}
Suppose that $N\geq1$, $2<p<\min\{4, 2^*\}$ and conditions $(D1)-(D2)$ hold.  If $\int_\Omega g\phi_1^pdx<0$, then for each $0<a<a_0(p)$ and $0<\lambda<\lambda_1(f)+\overline{\delta}_a$, there exists $e_0\in H^1_0(\Omega)$ such that
\begin{equation*}
\|e_0\|>\overline{\rho}_{a,\lambda}\quad\text{and}\quad \inf_{\|u\|=\overline{\rho}_{a,\lambda}}J_{a,\lambda}(u)>0>J_{a, \lambda}(e_0),
\end{equation*}
where $a_0(p)>0$ is as in $\eqref{a0}$.
\end{Lemma}

\begin{proof}
We firstly show that for each $0<\lambda <\lambda _{1}(f)+\overline{\delta}_{a}$, we
have $\inf_{\Vert u\Vert =\overline{\rho}_{a,\lambda }}J_{a,\lambda }(u)>0$. Now, we
separate this part into two cases:\newline
Case $(I):0<\lambda <\lambda _{1}(f)$. By $\eqref{r},$ $\eqref{u5}$ and
condition $(D2)$, we deduce that for $\Vert u\Vert =\overline{\rho }%
_{a,\lambda },$
\begin{align*}
J_{a,\lambda }(u)& \geq \frac{a}{4}\Vert u\Vert ^{4}+\frac{1}{2}\left( 1-%
\frac{\lambda }{\lambda _{1}(f)}\right) \Vert u\Vert ^{2}-\frac{\Vert g\Vert
_{\infty }}{pS_{p}^{p}}\Vert u\Vert ^{p} \\
& \geq \frac{a}{4}\overline{\rho }_{a,\lambda }^{4}+\frac{1}{4}\left( 1-%
\frac{\lambda }{\lambda _{1}(f)}\right) \overline{\rho }_{a,\lambda }^{2}>0.
\end{align*}%
Case $(II):\lambda _{1}(f)\leq \lambda < \lambda _{1}(f)+\overline{\delta
}_{a}.$ Repeating the same process in $\eqref{J1}-\eqref{J1'}$, we deduce
that
\begin{align}
\ & \quad J_{a,\lambda }(u)  \notag \\
\ & \geq \frac{a}{4}\Vert u\Vert ^{4}-|\theta _{1}|\Vert u\Vert ^{2}+\Theta
\Vert w\Vert ^{2}-\frac{1}{p}\int_{\Omega }g|t\phi _{1}|^{p}dx-\frac{1}{p}%
\left( \int_{\Omega }g(|t\phi _{1}+w|^{p}-|t\phi _{1}|^{p})dx\right)   \notag
\\
\ & \geq \frac{a}{4}\Vert u\Vert ^{4}-|\theta _{1}|\Vert u\Vert ^{2}+\Theta
\Vert w\Vert ^{2}+\frac{|\int_{\Omega }g\phi _{1}^{p}dx|}{2p}|t|^{p}-\frac{%
2^{p-2}\Vert g\Vert _{\infty }(1+pB^{p})}{pB^{p}S_{p}^{p}}\Vert w\Vert ^{p}
\notag \\
\ & \geq \frac{a}{4}\Vert u\Vert ^{4}-|\theta _{1}|\Vert u\Vert ^{2}+\Theta
\Vert w\Vert ^{2}+\frac{|\int_{\Omega }g\phi _{1}^{p}dx|}{2p\lambda _{1}(f)^{%
\frac{p}{2}}}\left( \frac{\Vert u\Vert ^{p}}{2^{\frac{p}{2}-1}}-\Vert w\Vert
^{p}\right) -\frac{2^{p-2}\Vert g\Vert _{\infty }(1+pB^{p})}{pB^{p}S_{p}^{p}}%
\Vert w\Vert ^{p}  \notag \\
\ & =-|\theta _{1}|\Vert u\Vert ^{2}+\frac{|\int_{\Omega }g\phi _{1}^{p}dx|}{%
2^{\frac{p}{2}}p\lambda _{1}(f)^{\frac{p}{2}}}\Vert u\Vert ^{p}+\frac{a}{4}%
\Vert u\Vert ^{4}  \notag \\
\ & \quad \quad +\Vert w\Vert ^{2}\left( \Theta -\left( \frac{|\int_{\Omega
}g\phi _{1}^{p}dx|}{2p\lambda _{1}(f)^{\frac{p}{2}}}+\frac{2^{p-2}\Vert
g\Vert _{\infty }(1+pB^{p})}{pB^{p}S_{p}^{p}}\right) \Vert w\Vert
^{p-2}\right) ,  \label{J31}
\end{align}%
where $\theta _{1}$ and $\Theta $ are as in $\eqref{n}$, $B$ is as in $%
\eqref{B}$, and $u=t\phi _{1}+w$ with $t\in \mathbb{R}$, $w\in
H_{0}^{1}(\Omega )$ and $\int_{\Omega }\nabla w\nabla \phi _{1}dx=0$. Since $%
\lambda \geq \lambda _{1}(f)$, we have
\begin{equation*}
\Theta \geq \frac{\lambda _{2}(f)-\lambda _{1}(f)}{2\lambda _{2}(f)}>0.
\end{equation*}%
Then by $\left( \ref{J21}\right) ,$ we have
\begin{equation}
\Theta -\left( \frac{|\int_{\Omega }g\phi _{1}^{p}dx|}{2p\lambda _{1}(f)^{%
\frac{p}{2}}}+\frac{2^{p-2}\Vert g\Vert _{\infty }(1+pB^{p})}{pB^{p}S_{p}^{p}%
}\right) \Vert w\Vert ^{p-2}\geq 0\quad \text{for all $0\leq \Vert w\Vert $}%
\leq \overline{\rho }_{a,\lambda }.  \label{J32}
\end{equation}%
Moreover,
\begin{equation}
|\theta _{1}|=\left\vert \frac{1}{2}\left( 1-\frac{\lambda }{\lambda _{1}(f)}%
\right) \right\vert \leq \frac{|\int_{\Omega }g\phi _{1}^{p}dx|}{2^{\frac{p}{%
2}+1}p\lambda _{1}(f)^{\frac{p}{2}}}\overline\rho _{a,\lambda }^{p}\quad \text{for
every $\lambda _{1}(f)\leq \lambda <\lambda _{1}(f)+\overline{\delta }_{a}$}.
\label{J33}
\end{equation}%
It follows from $\eqref{J31}-\eqref{J33},$
\begin{equation*}
J_{a,\lambda }(u)\geq \frac{|\int_{\Omega }g\phi _{1}^{p}dx|}{2^{\frac{p}{2}%
+1}p\lambda _{1}(f)^{\frac{p}{2}}}\overline{\rho }_{a,\lambda }^{p}+\frac{a}{%
4}\overline{\rho }_{a,\lambda }^{4}>0
\end{equation*}%
for every $\lambda _{1}(f)\leq \lambda <\lambda _{1}(f)+\overline{\delta }%
_{a}$ and $\Vert u\Vert =\overline{\rho }_{a,\lambda }$. Consequently, for
each $a>0$, we have $\inf_{\Vert u\Vert =\overline{\rho }_{a,\lambda
}}J_{a,\lambda }(u)>0$ for every $0<\lambda <\lambda _{1}(f)+\overline{%
\delta }_{a}.$

Next, we show that there exists $e_{0}\in H_{0}^{1}(\Omega )$ such that $%
\Vert e_{0}\Vert >\overline{\rho}_{a,\lambda }$ and $J_{a,\lambda }(e_{0})<0$. By
Proposition \ref{P:Ga-p}, there exists $\phi _{g}\in H_{0}^{1}(\Omega )$
such that $\Gamma _{p}\Vert \phi _{g}\Vert ^{p}=\int_{\Omega }g|\phi
_{g}|^{p}dx,$ $\int_{\Omega }f\phi _{g}^{2}dx\geq 0$ and%
\begin{equation}
0<a<a_{0}(p):=(2p-4)\left( 4-p\right) ^{\frac{4-p}{p-2}}\left( \frac{%
\int_{\Omega }g|\phi _{g}|^{p}dx}{p\Vert \phi _{g}\Vert ^{p}}\right) ^{\frac{%
2}{p-2}}.  \label{A1}
\end{equation}%
Let
\begin{equation*}
t_{a}=\left( \frac{(2p-4)\Gamma _{p}}{ap}\right) ^{\frac{1}{4-p}}\Vert \phi
_{g}\Vert ^{-1}.
\end{equation*}%
Then by $\left( \ref{J21}\right) $ and $\left( \ref{A1}\right) ,$ we can deduce that $\Vert
t_{a}\phi _{g}\Vert >\overline{\rho }_{a,\lambda }$ and
\begin{align*}
J_{a,\lambda }(t_{a}\phi _{g})& =\frac{t_{a}^{2}}{2}\left( \Vert \phi
_{g}\Vert ^{2}-\lambda \int_{\Omega }f\phi _{g}^{2}dx\right)
+t_{a}^{p}\left( \frac{a}{4}\Vert \phi _{g}\Vert ^{4}t_{a}^{4-p}-\frac{%
\int_{\Omega }g|\phi _{g}|^{p}dx}{p}\right)  \\
\ & =\frac{t_{a}^{2}}{2}\left( \Vert \phi _{g}\Vert ^{2}-\lambda
\int_{\Omega }f\phi _{g}^{2}dx-(4-p)\left( \frac{\int_{\Omega }g|\phi
_{g}|^{p}dx}{p}\right) \left( \frac{(2p-4)\int_{\Omega }g|\phi _{g}|^{p}dx}{%
ap\Vert \phi _{g}\Vert ^{4}}\right) ^{\frac{p-2}{4-p}}\right)  \\
\ & =\frac{t_{a}^{2}}{2}\left( \Vert \phi _{g}\Vert ^{2}-\Vert \phi
_{g}\Vert ^{2}(4-p)\left( \frac{\int_{\Omega }g|\phi _{g}|^{p}dx}{p\Vert
\phi _{g}\Vert ^{p}}\right) ^{\frac{2}{4-p}}\left( \frac{2p-4}{a}\right) ^{%
\frac{p-2}{4-p}}-\lambda \int_{\Omega }f\phi _{g}^{2}dx\right)  \\
& <0\text{ for all }0<a<a_{0}(p).
\end{align*}%
This completes the proof.
\end{proof}

In what follows, we study the case when $a\geq a_0(p)$. It is necessary to consider the following which is defined in \eqref{lam:a-}:
\begin{equation*}
\lambda_a^-:=\inf_{u\in S}\frac{\|u\|^2}{\int_\Omega fu^2dx}\left(1-(4-p)\left(\frac{\int_\Omega g|u|^pdx}{p\|u\|^p}\right)^{\frac{2}{4-p}}\left(\frac{2p-4}{a}\right)^{\frac{p-2}{4-p}}\right)\quad\text{for $a\geq a_0(p)$}. 
\end{equation*}

\begin{Prop}\label{P:lam:a-}
Suppose that $N\geq1$, $2<p<\min\{4, 2^*\}$ and conditions $(D1)-(D2)$ hold. Then $\lambda_a^-$ is nonnegative, continuous and increasing.
Furthermore, if $\int_\Omega g\phi_1^pdx<0$, then
$$\lim_{a\to\infty}\lambda_a^->\lambda_1(f).$$
\end{Prop}

\begin{proof}
It is easy to show that $\lambda_a^-$ is nonnegative, continuous and increasing, and thus we omit this part.
It is clear that
\begin{equation*}
\lambda_a^-=\inf_{u\in S'}\|u\|^2\left(1-(4-p)\left(\frac{\int_\Omega g|u|^pdx}{p\|u\|^p}\right)^{\frac{2}{4-p}}\left(\frac{2p-4}{a}\right)^{\frac{p-2}{4-p}}\right),
\end{equation*}
where
$$S'=\left\{u\in H^1_0(\Omega) : \int_\Omega fu^2dx=1, \int_\Omega g|u|^pdx>0\right\}.$$
Thus, by $\eqref{R1.2},$ we have
\begin{equation*}
\tilde\Lambda:=\inf_{u\in S'}\|u\|^2\geq\lambda_1(f).
\end{equation*}
We now claim that
\begin{equation}\label{inf1}
\tilde\Lambda>\lambda_1(f).
\end{equation}
Suppose that this claim is false. Then there exists a sequence $\{u_n\}$ with $\int_\Omega fu_n^2dx=1$ and $\int_\Omega g|u_n|^pdx>0$ such that
$$\lim_{n\to\infty}\|u_n\|^2=\lambda_1(f).$$
Clearly, $\{u_n\}$ is bounded, and thus there exist a subsequence $\{u_n\}$ and $u_0\in H^1_0(\Omega)$ such that
$u_n \rightharpoonup u_0$ in $H^1_0(\Omega)$ and
\begin{equation}\label{.}
u_n\to u_0\quad \text{in}\quad L^r(\Omega)\quad \text{for all}\quad 1\leq r<2^*.
\end{equation}
By conditions $(D1)-(D2)$ and $\eqref{.}$, we have
\begin{equation}\label{..}
\int_\Omega fu_0^2dx=\lim_{n\to\infty}\int_\Omega fu_n^2dx=1
\end{equation}
and
\begin{equation}\label{...}
\int_\Omega g|u_0|^pdx=\lim_{n\to\infty}\int_\Omega g|u_n|^pdx\geq0.
\end{equation}
Next, we show that $u_n\to u_0$ in $H^1_0(\Omega)$. If not, then $\|u_0\|^2<\liminf_{n\to\infty}\|u_n\|^2=\lambda_1(f)$, which is impossible due to $\eqref{R1.2}$ and $\eqref{..}$. Thus, we have $\|u_0\|^2=\lambda_1(f)$, which indicates that $u_0=\pm\phi_1$. Then by \eqref{...}, we obtain
$\int_\Omega g\phi_1^p\geq0$, which is a contradiction. Hence this claim is true.
Moreover, by $\eqref{inf1},$ we can deduce that
$$\lim_{a\to\infty}\lambda_a^-\geq\lim_{a\to\infty}\tilde\Lambda\left(1-(4-p)\left(\frac{\Gamma_p}{p}\right)^{\frac{2}{4-p}}\left(\frac{2p-4}{a}\right)^{\frac{p-2}{4-p}}\right)=\tilde\Lambda>\lambda_1(f).$$
This completes the proof.
\end{proof}

\begin{Prop}\label{P:lam:a2}
Suppose that $N\geq1$, $2<p<\min\{4, 2^*\}$ and conditions $(D1)-(D3)$ hold. Then $\lambda_{a_0(p)}^-=0$. Furthermore, if $\int_\Omega g\phi_1^p<0$, then there is a number ${\bf A}>a_0(p)$ such that $0\leq\lambda_a^-<\lambda_1(f)$ for all $a_0(p)\leq a<{\bf A}$.
\end{Prop}

\begin{proof}
By Proposition \ref{P:Ga-p}, there exists $u_0\in H^1_0(\Omega)$ with $\int_\Omega fu_0^2dx\geq 0$ such that $\Gamma_p\|u_0\|^p=\int_\Omega g|u_0|^pdx$. Then by condition $(D3),$ we conclude that $\int_\Omega fu_0^2dx>0$.
Hence
$$0\leq\lambda_{a_0(p)}^-\leq \frac{\|u_0\|^2}{\int_\Omega fu_0^2dx}\left(1-(4-p)\left(\frac{\int_\Omega g|u_0|^pdx}{p\|u_0\|^p}\right)^{\frac{2}{4-p}}\left(\frac{2p-4}{a_0(p)}\right)^{\frac{p-2}{4-p}}\right)=0.$$
Then by Proposition \ref{P:lam:a-}, if $\int_\Omega g\phi_1^pdx<0$, we can deduce that there exists one number ${\bf A}>a_0(p)$ such that $0\leq\lambda_a^-<\lambda_1(f)$ for every $a_0(p)\leq a<{\bf A}$.
This completes the proof.
\end{proof}

\begin{Remark}
By Proposition \ref{P:lam:a2} and the definition of $\lambda_a^-$, for every $a_0(p)\leq a<{\bf A}$ and $\lambda_a^-<\lambda<\lambda_1(f)$, there exists $\varphi _{a,\lambda}\in S$ such that
\begin{equation}\label{var-1}
\lambda _{a}^{-}\leq\frac{\Vert \varphi _{a,\lambda}\Vert^{2}}{\int_{\Omega }f\varphi _{a,\lambda}^{2}dx}\left( 1-(4-p)\left( \frac{\int_{\Omega}g|\varphi _{a,\lambda}|^{p}dx}{p\Vert \varphi _{a,\lambda}\Vert ^{p}}\right) ^{\frac{2}{4-p}}\left( \frac{2p-4}{a}\right) ^{\frac{p-2}{4-p}}\right) <\lambda.
\end{equation}
Moreover, for all $\lambda\geq\lambda_1(f)$, there exists $\varphi_a \in S$ which is independent of $\lambda$ such that
\begin{equation}\label{var-2}
\lambda _{a}^{-}\leq\frac{\Vert \varphi _{a}\Vert^{2}}{\int_{\Omega }f\varphi _{a}^{2}dx}\left( 1-(4-p)\left( \frac{\int_{\Omega}g|\varphi _{a}|^{p}dx}{p\Vert \varphi _{a}\Vert ^{p}}\right) ^{\frac{2}{4-p}}\left( \frac{2p-4}{a}\right) ^{\frac{p-2}{4-p}}\right) <\lambda_1(f).
\end{equation}
\end{Remark}

For $a\geq a_0(p)$ and $\lambda >\lambda_a^-,$ let%
\begin{equation}
\widehat{\delta}_{a}=\frac{\left\vert \int_{\Omega}g\phi_{1}^{p}dx\right\vert }{2^{\frac{p}{2}}p\lambda _{1}(f)^{\frac{p-2}{2}}}\min
\left\{\rho _{0}^{p-2},\rho _{2, a}^{p-2}\right\}   \label{J22}
\end{equation}%
and
\begin{equation}
\widehat{\rho }_{a,\lambda }=\left\{
\begin{array}{ll}
\min \left\{\rho _{\lambda }, \rho _{a, \lambda}\right\} , & \text{ for }\lambda
_{a}^{-}<\lambda <\lambda _{1}\left( f\right), \\
\min \left\{ \rho _{0},\rho _{2, a}\right\} , & \text{ for }\lambda
_{1}\left( f\right) \leq \lambda <\lambda _{1}\left( f\right) +\widehat{%
\delta }_{a},
\end{array}%
\right.   \label{J23}
\end{equation}%
where $\rho_0$ and $\rho_\lambda$ are as in \eqref{rho-0} and \eqref{rho-lam}, respectively, and
\begin{eqnarray*}
\rho _{a, \lambda} &=&\left( \frac{(p-2)\int_{\Omega }g|\varphi _{a,\lambda}|^{p}dx}{ap\Vert
\varphi _{a,\lambda}\Vert ^{p}}\right) ^{\frac{1}{4-p}}\quad\text{with $\varphi_{a,\lambda}$ as in \eqref{var-1}}, \\
\rho _{2, a} &=&\left( \frac{(p-2)\int_{\Omega }g|\varphi _{a}|^{p}dx}{ap\Vert
\varphi _{a}\Vert ^{p}}\right) ^{\frac{1}{4-p}}\quad\text{with $\varphi_{a}$ as in \eqref{var-2}}.
\end{eqnarray*}%
Then we have the following result.

\begin{Lemma}\label{MP:p<4-2}
Suppose that $N\geq 1,$ $2<p<\min \{4,2^{\ast }\}$ and
conditions $(D1)-(D3)$ hold. Let $\mathbf{A>0}$ be as in Proposition \ref%
{P:lam:a2} and $\int_{\Omega }g\phi _{1}^{p}dx<0.$ Then
for each $a_{0}(p)\leq a<\mathbf{A}$ and $\lambda _{a}^{-}<\lambda <\lambda
_{1}(f)+\widehat{\delta }_{a}$, there exists $e_{0}\in H_{0}^{1}(\Omega )$
such that
\begin{equation*}
\Vert e_{0}\Vert >\widehat{\rho }_{a,\lambda }\quad \text{and}\quad
\inf_{\Vert u\Vert =\widehat{\rho }_{a,\lambda }}J_{a,\lambda
}(u)>0>J_{a,\lambda }(e_{0}).
\end{equation*}
\end{Lemma}

\begin{proof}
We separate the proof into two cases.\newline
Case $(I):\lambda _{a}^{-}<\lambda <\lambda _{1}(f)$. By $\eqref{r},$ $%
\eqref{u5}$ and condition $(D2)$, we deduce that for $\Vert u\Vert =\widehat{%
\rho }_{a,\lambda },$%
\begin{align*}
J_{a,\lambda }(u)& \geq \frac{a}{4}\Vert u\Vert ^{4}+\frac{1}{2}\left( 1-%
\frac{\lambda }{\lambda _{1}(f)}\right) \Vert u\Vert ^{2}-\frac{\Vert g\Vert
_{\infty }}{pS_{p}^{p}}\Vert u\Vert ^{p} \\
& \geq \frac{a}{4}\widehat{\rho }_{a,\lambda }^{4}+\frac{1}{4}\left( 1-\frac{%
\lambda }{\lambda _{1}(f)}\right) \widehat{\rho }_{a,\lambda }^{2}>0.
\end{align*}%
Let
\begin{equation*}
t_{a,\lambda }:=\left( \frac{(2p-4)\int_{\Omega }g|\varphi _{a, \lambda }|^{p}dx}{%
ap\Vert \varphi _{a, \lambda }\Vert ^{4}}\right) ^{\frac{1}{4-p}}.
\end{equation*}%
Then by \eqref{var-1} and $\left( \ref{J23}\right) ,$ we deduce that $\Vert t_{a,\lambda }\varphi _{a, \lambda}\Vert >\widehat{\rho }_{a,\lambda }$ and
\begin{align*}
J_{a,\lambda }(t_{a,\lambda }\varphi _{a, \lambda })& =\frac{t_{a,\lambda }^{2}}{2}%
\left( \Vert \varphi _{a, \lambda }\Vert ^{2}-\lambda \int_{\Omega }f\varphi
_{a, \lambda }^{2}dx\right) +t_{a,\lambda }^{p}\left( \frac{a}{4}\Vert \varphi
_{a, \lambda }\Vert ^{4}t_{a,\lambda }^{4-p}-\frac{\int_{\Omega }g|\varphi
_{a,\lambda }|^{p}dx}{p}\right)  \\
& =\frac{t_{a,\lambda }^{2}}{2}\left( \Vert \varphi _{a,\lambda }\Vert
^{2}-\lambda \int_{\Omega }f\varphi _{a,\lambda }^{2}dx-\frac{(4-p)\int_{\Omega
}g|\varphi _{a,\lambda }|^{p}dx}{p}\left( \frac{(2p-4)\int_{\Omega }g|\varphi
_{a,\lambda }|^{p}dx}{ap\Vert \varphi _{a,\lambda }\Vert ^{4}}\right) ^{\frac{p-2}{%
4-p}}\right)  \\
& <0.
\end{align*}%
Case $(II): \lambda _{1}(f)\leq\lambda<\lambda_1(f)+\widehat{\delta }_{a}$. Repeating the same process in $%
\eqref{J31}-\eqref{J33}$, we can deduce that for each $\lambda _{1}(f)\leq
\lambda <\lambda _{1}(f)+\widehat{\delta }_{a}$ and $\Vert u\Vert =\widehat{%
\rho }_{a,\lambda },$ we have%
\begin{equation*}
J_{a,\lambda }(u)\geq \frac{\left\vert \int_{\Omega }g\phi
_{1}^{p}dx\right\vert }{2^{\frac{p}{2}+1}p\lambda _{1}(f)^{\frac{p}{2}}}%
\widehat{\rho }_{a,\lambda }^{p}+\frac{a}{4}\widehat{\rho }_{a,\lambda }^{4}.
\end{equation*}%
Let
\begin{equation*}
t_{a}=:\left( \frac{(2p-4)\int_{\Omega }g|\varphi _{a}|^{p}dx}{ap\Vert \varphi
_{a}\Vert ^{4}}\right) ^{\frac{1}{4-p}}.
\end{equation*}%
Then by \eqref{var-2} and $\left( \ref{J23}\right) ,$ we deduce that $\Vert t_{a}\varphi _{a}\Vert >\widehat\rho
_{a,\lambda }$ and%
\begin{align*}
J_{a,\lambda }(t_{a}\varphi _{a})& =\frac{t_{a}^{2}}{2}\left( \Vert \varphi
_{a}\Vert ^{2}-\lambda \int_{\Omega }f\varphi _{a}^{2}dx\right)
+t_{a}^{p}\left( \frac{a}{4}\Vert \varphi _{a}\Vert ^{4}t_{a}^{4-p}-\frac{%
\int_{\Omega }g|\varphi _{a}|^{p}dx}{p}\right)  \\
& =\frac{t_{a}^{2}}{2}\left( \Vert \varphi _{a}\Vert ^{2}-\lambda \int_{\Omega
}f\varphi _{a}^{2}dx-\frac{(4-p)\int_{\Omega }g|\varphi _{a}|^{p}dx}{p}\left(
\frac{(2p-4)\int_{\Omega }g|\varphi _{a}|^{p}dx}{ap\Vert \varphi _{a}\Vert ^{4}}%
\right) ^{\frac{p-2}{4-p}}\right)  \\
& <0.
\end{align*}%
This complete the proof.
\end{proof}

By Lemmas \ref{MP:p<4-1} and \ref{MP:p<4-2}, the mountain pass geometry of $J_{a,\lambda}$ appears for $\lambda$ in a neighbourhood that contains $\lambda_1(f)$ whenever $0<a<{\bf A}$. The corresponding mountain pass type solution is presented in Figure \ref{fg:p3-} (a)-(b) which is the middle branch of each curve.

Next, we prove that the functional $J_{a,\lambda}$ satisfies the $(PS)_\alpha$-condition.

\begin{Lemma}\label{LPS}
Suppose that $N\geq1$, $2<p<2^*$ and conditions $(D1)-(D2)$ hold. If the $(PS)_\alpha$-sequence for $J_{a,\lambda}$ is bounded, then it has a convergent subsequence.
\end{Lemma}

\begin{proof}
Let $\{u_n\}\subset H^1_0(\Omega)$ be a bounded $(PS)_\alpha$-sequence for $J_{a,\lambda}$. Then there exist a subsequence $\{u_n\}$ and $u_0\in H^1_0(\Omega)$ such that $u_n \rightharpoonup u_0$ in $H^1_0(\Omega)$ and
\begin{equation}\label{2.7}
u_n\to u_0\quad \text{in}\quad L^r(\Omega)\quad \text{for all}\quad 1\leq r<2^*.
\end{equation}
Since $\langle J_{a,\lambda}'(u_n), \varphi\rangle=o(1)$ for all $\varphi\in H^1_0(\Omega)$, we have
\begin{align*}
o(1)&=\langle J_{a,\lambda}'(u_n), u_n-u_0\rangle\\
\ &=\left(a\|u_n\|^2+1\right)\int_\Omega\nabla u_n\nabla(u_n-u_0)dx-\lambda\int_\Omega fu_n(u_n-u_0)dx-\int_\Omega g|u_n|^{p-2}u_n(u_n-u_0)dx.
\end{align*}
By conditions $(D1)-(D2)$ and $\eqref{2.7}$, we have
$$\int_\Omega fu_n(u_n-u_0)dx\to 0\quad\text{and}\quad\int_\Omega g|u_n|^{p-2}u_n(u_n-u_0)dx\to 0.$$
Thus,
$$\left(a\|u_n\|^2+1\right)\int_\Omega\nabla u_n\nabla(u_n-u_0)dx=o(1),$$
which implies that $u_n\to u_0$ in $H^1_0(\Omega)$. This completes the proof.
\end{proof}

\section{The Proof of Theorem \ref{T:p>4}}

In this section, we present the proof of Theorem \ref{T:p>4} and a discussion on the asymptotic behaviour of the lower branch solutions of case $(ii)$ in the limit $\lambda\to\lambda_1^+(f)$ will also be given after the proof of the main result.

{\bf The proof of Theorem \ref{T:p>4}:} From Lemma \ref{MP:p>4}, for each $a>0$ there exists $\delta_a>0$, with $\delta_a\to\infty$ as $a\to\infty$, such that the functional $J_{a,\lambda}$ has the mountain pass geometry for every $0<\lambda<\lambda_1(f)+\delta_a$. Then by Lemma \ref{LPS}, if the $(PS)_\alpha$-sequence for $J_{a,\lambda}$ is bounded, then there exists $u^+\in H^1_0(\Omega)$ such that $J_{a,\lambda}(u^+)>0$ and $J'_{a,\lambda}(u^+)=0$. Since $J_{a,\lambda}(u^+)=J_{a,\lambda}(|u^+|)$, we assume without loss of generality that $u^+>0$ in $\Omega$.
Thus, we need to show the boundedness of the $(PS)_\alpha$-sequence for $J_{a,\lambda}$.

Suppose that the $(PS)_\alpha$-sequence $\{u_n\}$ is unbounded. Then, as $n\to\infty$, we have $\|u_n\|\to\infty$,
\begin{equation}\label{bd2}
J_{a,\lambda}(u_n)=\frac a4\|u_n\|^4+\frac 12\|u_n\|^2-\frac{\lambda}{2}\int_\Omega fu_n^2dx-\frac1p\int_\Omega g|u_n|^pdx \to \alpha,
\end{equation}
and
\begin{equation}\label{bd3}
\langle J'_{a,\lambda}(u_n), u_n\rangle=a\|u_n\|^4+\|u_n\|^2-\lambda\int_\Omega fu_n^2 dx-\int_\Omega g|u_n|^p dx \to 0.
\end{equation}
Let $v_n=\frac{u_n}{\|u_n\|}$, then $\|v_n\|=1$ and thus there exist a subsequence $\{v_n\}$ and $v_0\in H^1_0(\Omega)$ such that $v_n \rightharpoonup v_0$ in $H^1_0(\Omega)$ and
\begin{equation}\label{dd1}
v_n\to v_0\ \text{in $L^r(\Omega)$ for all $1\leq r<2^*$.}
\end{equation}
By  \eqref{dd1} and condition $(D1)$, we have
\begin{equation}\label{d1}
\lim_{n\to\infty}\int_\Omega fv_n^2dx=\int_\Omega fv_0^2dx.
\end{equation}
Dividing \eqref{bd2} and \eqref{bd3} by $\|u_n\|^4$, because of \eqref{d1}, we obtain that
\begin{equation}\label{bd4}
\frac a4-\frac{\|u_n\|^{p-4}}{p}\int_\Omega g|v_n|^pdx \to 0
\end{equation}
and
\begin{equation}\label{bd6}
a-\|u_n\|^{p-4}\int_\Omega g|v_n|^p dx \to 0.
\end{equation}
This gives a contradiction. Hence $\{u_n\}$ is bounded and we have obtained a positive solution whenever $0<\lambda<\lambda_1(f)+\delta_a$.
Here we have completed the proofs of $(i)$ and the relevant part of $(ii)$.

To complete the proof of case $(ii)$, we consider the infimum of $J_{a,\lambda}$ on the closed ball $B_{\rho_a}:=\{u\in H^1_0(\Omega) : \|u\|\leq\rho_a\}$ with $\rho_a$ as in Lemma \ref{MP:p>4} $(i)$ and $(ii)$.
Set
$$\beta_0:=\inf_{\|u\|\leq\rho_a}J_{a,\lambda}(u).$$
For any $t>0$,
\begin{equation*}
J_{a,\lambda}(t\phi_1)= -\frac {\lambda-\lambda_1(f)}{2}t^2+\frac {a\lambda_1(f)^2}{4}t^4-\frac{\int_\Omega g\phi_1^pdx}{p}t^p.
\end{equation*}
Thus, for each $\lambda>\lambda_1(f)$, there exists $t_0>0$ such that $\|t_0\phi_1\|\leq\rho_a$ and $J_{a,\lambda}(t_0\phi_1)<0$. Moreover,
\begin{align*}
J_{a,\lambda}(u)\geq-\frac{\lambda}{2\lambda_1(f)}\|u\|^2-\frac{\|g\|_\infty}{pS_p^p}\|u\|^p
\geq -\frac12\frac{\lambda}{\lambda_1(f)}\rho_a^2-\frac{\|g\|_\infty}{pS_p^p}\rho_a^p,
\end{align*}
giving $-\infty<\beta_0<0$. By the Ekeland variational principle \cite{E}, there exists a $(PS)_{\beta_0}$-sequence $\{u_n\}\subset B_{\rho_0}$. Then by Lemma \ref{LPS}, there exists $u^-\in H^1_0(\Omega)$ such that $J_{a,\lambda}(u^-)=\beta_0$ and $J'_{a,\lambda}(u^-)=0$. Since $J_{a,\lambda}(u^-)=J_{a,\lambda}(|u^-|)$, we assume without loss of generality that $u^->0$ in $\Omega$. This completes the proof of Theorem \ref{T:p>4}.

We now investigate the asymptotic behaviour of solutions $u^-$ obtained in Theorem \ref{T:p>4} $(ii)$ as $\lambda\to\lambda_1^+(f)$.

\begin{Theorem}\label{T3.1}
For each $a>0$, let $\lambda_n\to\lambda_1^+(f)$ and $u_n^-$ be the solution obtained in Theorem \ref{T:p>4} (ii) with $J_{a,\lambda_n}(u^-_n)<0$. Then, as $n\to\infty$, we have
$$(i)\ u^-_n\to0\quad\quad \text{and}\quad\quad  (ii)\ \frac{u^-_n}{\|u^-_n\|}\to k\phi_1\ \text{in}\ H^1_0(\Omega)\ \text{for some}\ k>0.$$
\end{Theorem}

\begin{proof}
$(i)$ According to the proof of Lemma \ref{MP:p>4}, $\rho_a$ is independent of $\lambda$. Thus, $\{u^-_n\}$ is bounded and so there exist a subsequence $\{u^-_n\}$ and $u^-_0\in H^1_0(\Omega)$ such that $u^-_n \rightharpoonup u^-_0$ in $H^1_0(\Omega)$ and
\begin{equation}\label{3.6}
u^-_n\to u^-_0 \text{ in $L^r(\Omega)$ for all $1\leq r<2^*$.}
\end{equation}
Then by \eqref{3.6} and condition $(D1)$, we have
$$\lim_{n\to\infty}\lambda_n\int_\Omega f(u^-_n)^2dx=\lambda_1(f)\int_\Omega f(u^-_0)^2dx.$$
It follows from $J_{a,\lambda_n}(u^-_n)<0$ and $\langle J_{a,\lambda_n}'(u^-_n), u^-_n\rangle=0$ that
\begin{equation}\label{3.7}
J_{a,\lambda_n}(u^-_n)=\left(\frac14-\frac1p\right)a\|u^-_n\|^4+\left(\frac12-\frac1p\right)\left(\|u^-_n\|^2-\lambda_n\int_\Omega f(u^-_n)^2dx\right)<0,
\end{equation}
which indicates that
\begin{equation}\label{3.8}
\|u^-_n\|^2-\lambda_n\int_\Omega f(u^-_n)^2dx<0.
\end{equation}
Suppose that $u^-_n\to u^-_0$ in $H^1_0(\Omega)$ does not hold. Then by \eqref{3.8},
$$\int_\Omega|\nabla u^-_0|^2dx-\lambda_1(f)\int_\Omega f(u^-_0)^2dx <\liminf_{n\to\infty}\left(\|u^-_n\|^2-\lambda_n\int_\Omega f(u^-_n)^2dx\right)\leq0,$$
which is impossible. Hence $u^-_n\to u^-_0$ and $\int_\Omega|\nabla u^-_0|^2dx-\lambda_1(f)\int_\Omega f(u^-_0)^2dx=0$. Then by $\eqref{3.7},$ we conclude $u^-_0=0$. Hence $u^-_n\to0$ in $H^1_0(\Omega)$.

$(ii)$ Let $v_n=\frac{u^-_n}{\|u^-_n\|}$, then there exist a subsequence $\{v_n\}$ and $v_0\in H^1_0(\Omega)$ such that $v_n \rightharpoonup v_0$ in $H^1_0(\Omega)$ and
\begin{equation}\label{3.9}
v_n\to v_0 \text{ in $L^r(\Omega)$ for all $1\leq r<2^*$.}
\end{equation}
By conditions $(D1)-(D2)$ and $\eqref{3.9}$, we have
\begin{equation}\label{3.10}
\lim_{n\to\infty}\lambda_n\int_\Omega fv_n^2dx=\lambda_1(f)\int_\Omega fv_0^2dx\quad\text{and}\quad \lim_{n\to\infty}\int_\Omega g|v_n|^pdx=\int_\Omega g|v_0|^pdx.
\end{equation}
 It follows from $\langle J_{a,\lambda_n}'(u^-_n), u^-_n\rangle=0$, $\|u^-_n\|\to 0$ and \eqref{3.10} that
$$\|v_n\|^2-\lambda_n\int_\Omega fv_n^2dx=\|u^-_n\|^{p-2}\int_\Omega gv_n^pdx-a\|u^-_n\|^2 \to 0.$$
Suppose that $v_n\to v_0$ does not hold. Then
$$\int_\Omega|\nabla v_0|^2dx-\lambda_1(f)\int_\Omega fv_0^2dx<\liminf_{n\to\infty}\left(\|v_n\|^2-\lambda_n\int_\Omega fv_n^2dx\right)=0,$$
which is impossible. Hence $v_n\to v_0$, and then $\|v_0\|=1$ and $\int_\Omega|\nabla v_0|^2dx-\lambda_1(f)\int_\Omega fv_0^2dx=0$. Thus, $v_0=k\phi_1$ for some $k>0$. This completes the proof.
\end{proof}

\section{The Proofs of Theorems \ref{T:p=4+} and \ref{T:p=4-}}

In this section, we give the proofs of Theorems \ref{T:p=4+} and \ref{T:p=4-} and we begin by proving the boundedness of Palais-Smale sequence for the functional $J_{a,\lambda}$ when $p=4$. The asymptotic behaviours of solutions will also be discussed at the end of the section.

\begin{Lemma}\label{BD4}
Suppose that $N=1,2,3$, $p=4$ and conditions $(D1)-(D2)$ hold. Then we have the following results.
\vspace{-8pt}
\begin{enumerate}[(i)]
\item For each $a>0$ and $0<\lambda<\lambda_1(f)$, the $(PS)_\alpha$-sequence $\{u_n\}$ for $J_{a, \lambda}$ is bounded.\itemsep -4pt
\item If $a>\lambda_1(f)^{-2}\int_\Omega g\phi_1^4dx$, then there exists $\delta_2>0$ such that the $(PS)_\alpha$-sequence $\{u_n\}$ for $J_{a, \lambda}$ is bounded whenever $\lambda_1(f)\leq\lambda<\lambda_1(f)+\delta_2$.
\end{enumerate}
\end{Lemma}

\begin{proof}
$(i)$ Since
\begin{equation}\label{bd41}
\frac a4\|u_n\|^4+\frac 12\|u_n\|^2-\frac{\lambda}{2}\int_\Omega fu_n^2dx-\frac14\int_\Omega gu_n^4dx \to \alpha,
\end{equation}
and
\begin{equation}\label{bd42}
\left( a\|u_n\|^2+1\right)\int_\Omega\nabla u_n\nabla\varphi dx-\lambda\int_\Omega fu_n\varphi dx-\int_\Omega gu_n^3\varphi dx \to 0
\end{equation}
for all $\varphi\in H^1_0(\Omega)$. By $\eqref{bd42},$ we also have
\begin{equation}\label{bd43}
a\|u_n\|^4+\|u_n\|^2-\lambda\int_\Omega fu_n^2dx-\int_\Omega gu_n^4dx=o(1).
\end{equation}
Combining $\eqref{bd41}$ and $\eqref{bd43},$ we deduce that
$$\left(1-\frac{\lambda}{\lambda_1(f)}\right)\|u_n\|^2+o(1)\leq \|u_n\|^2-\lambda\int_\Omega fu_n^2dx+o(1)\to 4\alpha,$$
which implies that $\{u_n\}$ is bounded.

$(ii)$ Suppose that the result is false. Then $\|u_n\|\to\infty$ for every $\lambda\geq\lambda_1(f)$. Let $v_n=\frac{u_n}{\|u_n\|}$, then $\|v_n\|=1$ and thus there exist a subsequence $\{v_n\}$ and $v_0\in H^1_0(\Omega)$ such that $v_n \rightharpoonup v_0$ in $H^1_0(\Omega)$ and
\begin{equation}\label{5.1.1}
v_n\to v_0 \text{ in $L^r(\Omega)$ for all $1\leq r<2^*$.}
\end{equation}
By conditions $(D1)-(D2)$ and \eqref{5.1.1}, we have
$$\lim_{n\to\infty}\int_\Omega fv_n^2dx=\int_\Omega fv_0^2dx\quad\text{and}\quad\lim_{n\to\infty}\int_\Omega gv_n^4 dx=\int_\Omega gv_0^4 dx.$$
Combining $\eqref{bd41}$ and $\eqref{bd43},$ we have
\begin{equation}\label{bd44}
\|u_n\|^2-\lambda\int_\Omega fu_n^2dx+o(1)\to 4\alpha
\end{equation}
and
\begin{equation}\label{bd45}
a\|u_n\|^4-\int_\Omega gu_n^4dx+o(1)\to-4\alpha.
\end{equation}
Dividing $\eqref{bd44}$ by $\|u_n\|^2$ and \eqref{bd45} by $\|u_n\|^4$ then gives
$$1=\lambda\int_\Omega fv_0^2dx\quad\text{and}\quad a=\int_\Omega gv_0^4dx.$$
Dividing $\eqref{bd42}$ by $\|u_n\|^3$ and choosing $\varphi=v_0$, we have
$$a\|v_0\|^2=\int_\Omega gv_0^4dx.$$
Thus, $\|v_0\|=1$. Now, we define
$$\Lambda^0:=\left\{u \in H^1_0(\Omega) : \|u\|=1, \|u\|^2-\lambda\int_\Omega fu^2dx=0\right\},$$
and
$$\Theta^0:=\left\{u\in H^1_0(\Omega) : \|u\|=1, \int_\Omega gu^4dx-a\|u\|^4=0\right\}.$$
Clearly, $v_0\in \Lambda^0\cap\Theta^0$. Next, we claim that  there exists $\delta_2>0$ such that for every $\lambda_1(f)\leq\lambda<\lambda_1(f)+\delta_2$, it holds that $\Lambda^0\cap\Theta^0=\emptyset$ which is a contradiction, thus implying the conclusion of $(ii)$.

Suppose that this above claim is false. Then there exist two sequences $\{\lambda_n\}$ and $\{u_n\}\subset H^1_0(\Omega)$ with $\lambda_n\to \lambda_1^+(f)$, $\|u_n\|=1$,
\begin{equation}\label{bd31}
\|u_n\|^2-\lambda_n\int_\Omega fu_n^2dx=0,
\end{equation}
and
\begin{equation}\label{bd32}
\int_\Omega gu_n^4dx-a\|u_n\|^4=0.
\end{equation}
Since $\{u_n\}$ is bounded, there exist a subsequence $\{u_n\}$ and $u_0\in H^1_0(\Omega)$ such that $u_n \rightharpoonup u_0$ in $H^1_0(\Omega)$ and $u_n\to u_0$ in $L^r(\Omega)$ for all $1\leq r<2^*$. By conditions $(D1)-(D2),$ $\eqref{bd31}$ and $\eqref{bd32},$  we then have
\begin{equation}\label{bd33}
\lim_{n\to\infty}\lambda_n\int_\Omega fu_n^2dx=\lambda_1(f)\int_\Omega fu_0^2dx= 1
\end{equation}
and
$$\lim_{n\to\infty}\int_\Omega gu_n^4dx=\int_\Omega gu_0^4dx= a.$$
Now, we show that $u_n\to u_0$ in $H^1_0(\Omega)$. If not, it follows from $\eqref{bd31}$ and $\eqref{bd33}$ that
$$\int_\Omega|\nabla u_0|^2dx-\lambda_1(f)\int_\Omega fu_0^2dx <\liminf_{n\to\infty}\left(\|u_n\|^2-\lambda_n\int_\Omega fu_n^2dx\right)=0,$$
which is impossible. Hence $u_n\to u_0$ in $H^1_0(\Omega)$. It follows from \eqref{bd31} and \eqref{bd32} that
$$ (I)\ \|u_0\|^2=\lambda_1(f)\int_\Omega fu_0^2dx=1\quad;\quad (II) \int_\Omega gu_0^4dx=a.$$
Here $(I)$ gives that $u_0=\pm\lambda_1(f)^{-\frac12}\phi_1$ and then $\int_\Omega gu_0^4dx=\lambda_1(f)^2\int_\Omega g\phi_1^4dx=a$ by $(II)$, which is a contradiction since $a>\lambda_1(f)^{-2}\int_\Omega g\phi_1^4dx$. Thus, the claim above is true and this completes the proof.
\end{proof}

The following two theorems prove the results of Theorems \ref{T:p=4+} and \ref{T:p=4-}.

\begin{Theorem}\label{T4.2}
Suppose that $N=1,2,3$, $p=4$ and conditions $(D1)-(D2)$ hold. Then we have the following results.
\vspace{-8pt}
\begin{enumerate}[(i)]\itemsep-4pt
\item For each $0<a<\Gamma_0$ and $0<\lambda< \lambda_1(f)$, Equation $(K_{a,\lambda})$ has a positive solution $u^+$ with $J_{a, \lambda}(u^+)>0$.
\item For each $\max\{0, \lambda_1(f)^{-2}\int_\Omega g\phi_1^4dx\}<a<\Gamma_0$ and $\lambda=\lambda_1(f)$, Equation $(K_{a,\lambda})$ has a positive solution $u^+$ with $J_{a, \lambda}(u^+)>0$.
\item For each $\max\{0, \lambda_1(f)^{-2}\int_\Omega g\phi_1^4dx\}<a<\Gamma_0$, there exists $\delta_0>0$ such that Equation $(K_{a,\lambda})$ has two positive solutions $u^+$ and $u^-$ with $J_{a, \lambda}(u^-)<0<J_{a, \lambda}(u^+)$ whenever $\lambda_1(f)<\lambda< \lambda_1(f)+\delta_0$.
\end{enumerate}
\end{Theorem}

\begin{proof}
$(i)$ By Lemmas \ref{MP:p=4}, \ref{LPS} and \ref{BD4} $(i)$, for each $0<a<\Gamma_0$ and $0<\lambda<\lambda_1(f)$, there exists $u^+\in H^1_0(\Omega)$ such that $J_{a,\lambda}(u^+)>0$ and $J'_{a,\lambda}(u^+)=0$. Since $J_{a,\lambda}(u^+)=J_{a,\lambda}(|u^+|)$, which, without loss of generality, allows for the solution $u^+>0$ in $\Omega$.

$(ii)$-$(iii)$ The following argument proves $(ii)$ and the relevant part of $(iii)$. By Lemmas \ref{MP:p=4} $(ii)$-$(iii)$, \ref{LPS} and \ref{BD4} $(ii)$, for each $\max\{0, \lambda_1(f)^{-2}\int_\Omega g\phi_1^4dx\}<a<\Gamma_0$, there exists $\delta_0>0$ such that for every $\lambda_1(f)\leq\lambda<\lambda_1(f)+\delta_0$, there exists $u^+\in H^1_0(\Omega)$ such that $J_{a,\lambda}(u^+)>0$ and $J'_{a,\lambda}(u^+)=0$. Since $J_{a,\lambda}(u^+)=J_{a,\lambda}(|u^+|)$, we may take $u^+>0$ in $\Omega$. This concludes the proof here. 

To prove the remaining part of $(iii)$, we consider the infimum of $J_{a,\lambda}$ on the closed ball $B_{\rho_0}:=\{u\in H^1_0(\Omega) : \|u\|\leq\rho_0\}$ with $\rho_0$ as in \eqref{rho4}.
Set
$$\beta_0:=\inf_{\|u\|\leq\rho_0}J_{a,\lambda}(u).$$
For any $t>0$,
\begin{equation*}
J_{a,\lambda}(t\phi_1)= -\frac {\lambda-\lambda_1(f)}{2}t^2+\frac{a\lambda_1(f)^2-\int_\Omega g\phi_1^4dx}{4}t^4.
\end{equation*}
Clearly, there exists $t_0>0$ such that $\|t_0\phi_1\|\leq\rho_0$ and $J_{a,\lambda}(t_0\phi_1)<0$. Moreover,
\begin{align*}
J_{a,\lambda}(u)\geq-\frac 12\frac{\lambda}{\lambda_1(f)}\|u\|^2-\frac{\|g\|_\infty}{4S_4^4}\|u\|^4\geq -\frac{\lambda}{2\lambda_1(f)}\rho_0^2-\frac{\|g\|_\infty}{4S_4^4}\rho_0^4.
\end{align*}
Consequently, $-\infty<\beta_0<0$. By the Ekeland variational principle \cite{E}, there exists a $(PS)_{\beta_0}$-sequence $\{u_n\}\subset B_{\rho_0}$. Then by Lemma \ref{LPS}, there exists $u^-\in H^1_0(\Omega)$ such that $J_{a,\lambda}(u^-)=\beta_0$ and $J'_{a,\lambda}(u^-)=0$. Since $J_{a,\lambda}(u^-)=J_{a,\lambda}(|u^-|)$, this can be satisfied by taking $u^->0$ in $\Omega$.
This completes the proof.
\end{proof}

\begin{Theorem}\label{T43}
Suppose that $N=1,2,3$, $p=4$ and conditions $(D1)-(D2)$ hold. Then we have the following results.
\vspace{-6pt}
\begin{enumerate}[(i)]
\item For each $a\geq\Gamma_0$, Equation $(K_{a,\lambda})$ does not admit nontrivial solution whenever $0<\lambda\leq \lambda_1(f)$.\itemsep-4pt
\item For each $a>\Gamma_0$, Equation $(K_{a,\lambda})$ has a positive solution $u^-$ with $J_{a, \lambda}(u^-)<0$ whenever $\lambda>\lambda_1(f)$.
\end{enumerate}
\end{Theorem}

\begin{proof}
$(i)$ Suppose that the result is false. Then there exists a solution $u_0\neq 0$ such that
$$I_4(u_0):=\langle J_{a,\lambda}'(u_0), u_0\rangle=a\|u_0\|^4+\|u_0\|^2-\lambda\int_\Omega fu_0^2dx-\int_\Omega gu_0^4dx=0.$$
Since $a\geq\Gamma_0$ and $0<\lambda<\lambda_1(f)$ or  $a>\Gamma_0$ and $\lambda=\lambda_1(f)$, we have
$$I_4(u_0)\geq\left(1-\frac{\lambda}{\lambda_1(f)}\right)\|u_0\|^2+a\|u_0\|^4-\int_\Omega gu_0^4dx>0,$$
which is a contradiction.
For the rest case $a=\Gamma_0$ and $\lambda=\lambda_1(f)$, we can deduce that
$$(I)\ \|u_0\|^2=\lambda_1(f)\int_\Omega fu_0^2dx\quad\text{and}\quad (II)\ \Gamma_0\|u_0\|^4=\int_\Omega gu_0^4dx.$$
Here the result of $(I)$ implies that $u_0=k\phi_1$ for some $k\in\mathbb{R}$ whereas that of $(II)$ implies $k=0$; the results contradict each other. Hence Equation $(K_{a, \lambda})$ does not admit nontrivial solution for all $a\geq\Gamma_0$ and $0<\lambda\leq\lambda_1(f)$.

$(ii)$ By Young's inequality, we deduce that
\begin{align}\label{4.7}
J_{a, \lambda}(u)&=\frac 12\|u\|^2-\frac{\lambda}{2}\int_\Omega fu^2dx+\frac{a-\Gamma_0}{4}\|u\|^4+\frac14\left(\Gamma_0\|u\|^4-\int_\Omega gu^4dx\right)\notag\\
\ &\geq\frac{a-\Gamma_0}{4}\|u\|^4-\frac12\left(\frac{\lambda}{\lambda_1(f)}-1\right)\|u\|^2\notag\\
\ &\geq \frac{a-\Gamma_0}{8}\|u\|^4-\frac{1}{2(a-\Gamma_0)}\left(\frac{\lambda}{\lambda_1(f)}-1\right)^2.
\end{align}
Thus, $J_{a, \lambda}$ is coercive and bounded below for all $a>\Gamma_0$ and $\lambda>\lambda_1(f)$. Now, we consider the infimum of functional $J_{a,\lambda}$ on $H^1_0(\Omega)$. Set
$$\beta_0:=\inf_{u\in H^1_0(\Omega)}J_{a,\lambda}(u).$$
Since
\begin{equation*}
J_{a,\lambda}(t\phi_1)= -\frac {\lambda-\lambda_1(f)}{2}t^2+\frac{a\lambda_1(f)^2-\int_\Omega g\phi_1^4dx}{4}t^4  \text{  for } t>0.
\end{equation*}
This implies that there exists $t_0>0$ such that $J_{a,\lambda}(t_0\phi_1)<0.$ Hence $-\infty<\beta_0<0$. Then by the Ekeland variational principle \cite{E}, there exists a $(PS)_{\beta_0}$-sequence $\{u_n\}\subset H^1_0(\Omega)$. It follows from \eqref{4.7} that $\{u_n\}$ is bounded. Then by Lemma \ref{LPS}, there exists $u^-\in H^1_0(\Omega)$ such that $J_{a,\lambda}(u^-)=\beta_0$ and $J'_{a,\lambda}(u^-)=0$. With $J_{a,\lambda}(u^-)=J_{a,\lambda}(|u^-|)$, we may assume that $u^->0$ in $\Omega$.
This completes the proof.
\end{proof}

Next, we investigate the asymptotic behaviour of solutions.

\begin{Theorem}
Suppose that $\int_\Omega g\phi_1^4dx>0 $ and $ 0<a<\lambda_1(f)^{-2}\int_\Omega g\phi_1^4dx.$ Let $\lambda_n\to\lambda_1^-(f)$ and let $u^+_n$ be the solution obtained in Theorem \ref{T4.2} $(i)$ for Equation $(K_{a,\lambda_n})$. Then we have
$$(i)\ J_{a,\lambda_n}(u^+_n)\to0\quad\quad\text{and}\quad\quad (ii)\ u^+_n\to0.$$
\end{Theorem}

\begin{proof}
$(i)$ According to the Mountain Pass Theorem,
$$J_{a,\lambda_n}(u^+_n)=\inf_{\gamma\in \Gamma}\max_{t\in[0,1]}J_{a,\lambda_n}(\gamma(t))\quad\mbox{and}\quad
\Gamma:=\{\gamma\in C([0,1], H^1_0(\Omega)) : \gamma(0)=0, \gamma(1)=e_0\},$$
where $e_0$ is obtained in Lemma \ref{MP:p=4} $(i)$. Since $a<\lambda_1(f)^{-2}\int_\Omega g\phi_1^4dx$, we can choose $e_0=k\phi_1$ for some $k>0$. Then a direct calculation gives
\begin{align}\label{4.9}
0<J_{a,\lambda_n}(u^+_n)\leq\max_{t\in[0,1]} J_{a,\lambda_n}(tk\phi_1)=\frac{(\lambda_1(f)-\lambda_n)^3}{4\left(\int_\Omega g\phi_1^4dx-a\lambda_1(f)^2\right)}.
\end{align}
Hence it follows that $\lim_{n\to\infty}J_{a,\lambda_n}(u^+_n)=0$.

$(ii)$ By $\langle J_{a,\lambda_n}'(u^+_n), u^+_n\rangle=0$ and $\eqref{4.9},$ we have
\begin{align*}
\left(1-\frac{\lambda_n}{\lambda_1(f)}\right)\|u^+_n\|^2&\leq\|u^+_n\|^2-\lambda_n\int_\Omega f(u^+_n)^2dx\\
\ &=4J_{a,\lambda_n}(u^+_n)\leq \frac{(\lambda_1(f)-\lambda_n)^3}{\int_\Omega g\phi_1^4dx-a\lambda_1(f)^2},
\end{align*}
and thus
$$\|u^+_n\|^2\leq \frac{\lambda_1(f)(\lambda_1(f)-\lambda_n)^2}{\int_\Omega g\phi_1^4dx-a\lambda_1(f)^2}\longrightarrow 0\ \ \text{as}\ \ n\to\infty.$$
This completes the proof.
\end{proof}

\begin{Theorem}
Suppose that $\max\{0, \lambda_1(f)^{-2}\int_\Omega g\phi_1^pdx\}<a<\Gamma_0.$ Let $\lambda_n\to\lambda_1^+(f)$ and let $u_n^-$ be the solution obtained in Theorem \ref{T4.2} $(iii)$ with $J_{a,\lambda_n}(u^-_n)<0$. Then we have
$$(i)\ u^-_n\to0\quad\quad \text{and}\quad\quad  (ii)\ \frac{u^-_n}{\|u^-_n\|}\to k\phi_1\ \text{in}\ H^1_0(\Omega)\ \text{for some}\ k>0.$$
\end{Theorem}

\begin{proof}
$(i)$ From the proof of Theorem \ref{T4.2} $(iii)$, we have $\|u_n^-\|\leq\rho_0$ for all $n$, where $\rho_0$ is as \eqref{rho4}. Thus, there exist a subsequence $\{u^-_n\}$ and $u^-_0\in H^1_0(\Omega)$ such that $u^-_n \rightharpoonup u^-_0$ in $H^1_0(\Omega)$ and
\begin{equation}\label{5.12}
u^-_n\to u^-_0 \text{ in $L^r(\Omega)$ for all $1\leq r<2^*$.}
\end{equation}
By condition $(D1)$ and $\eqref{5.12},$ we have
\begin{equation}\label{5.13}
\lim_{n\to\infty}\lambda_n\int_\Omega f(u^-_n)^2dx=\lambda_1(f)\int_\Omega f(u^-_0)^2dx.
\end{equation}
It follows from $J_{a,\lambda_n}(u^-_n)<0$ and $\langle J_{a,\lambda_n}'(u^-_n), u^-_n\rangle=0$ that
\begin{equation}\label{5.14}
\int_\Omega g(u_n^-)^4dx-a\|u_n^-\|^4=\|u^-_n\|^2-\lambda_n\int_\Omega f(u^-_n)^2dx<0.
\end{equation}
Suppose that $u^-_n\to u^-_0$ does not hold. Then by $\eqref{5.13}$ and $\eqref{5.14},$ we have
$$\int_\Omega|\nabla u^-_0|^2dx-\lambda_1(f)\int_\Omega f(u^-_0)^2dx <\liminf_{n\to\infty}\left(\|u^-_n\|^2-\lambda_n\int_\Omega f(u^-_n)^2dx\right)\leq0,$$
which is impossible. Hence $u^-_n\to u^-_0$. It follows that
$$(I) \int_\Omega|\nabla u^-_0|^2dx-\lambda_1(f)\int_\Omega f(u^-_0)^2dx=0\quad\text{and}\quad (II) \int_\Omega g(u^-_0)^4dx-a\|u^-_0\|^4=0.$$
The result from $(I)$ implies $u^-_0=k\phi_1$ for some $k\geq0$ and from $(II)$ $k=0$ since $a>\lambda_1(f)^{-2}\int_\Omega g\phi_1^4dx$. Hence $u^-_n\to0$ in $H^1_0(\Omega)$.

$(ii)$ This part of the proof is identical to that of Theorem \ref{T3.1}$ (ii)$ and is not repeated here.
\end{proof}

\begin{Theorem}
Suppose that $0<a<\Gamma_0$ and $0<\lambda<\lambda_1(f).$ Let $u_a$ be the solution obtained in Theorem \ref{T4.2} $(i)$. Then $\|u_a\|\to\infty$ as $a\to\Gamma_0^-$.
\end{Theorem}

\begin{proof}
It is clear that $u_a$ satisfies the following equation
$$a\|u_a\|^4+\|u_a\|^2-\lambda\int_\Omega fu_a^2dx-\int_\Omega gu_a^4dx=0.$$
Thus,
\begin{align*}
\left(1-\frac{\lambda}{\lambda_1(f)}\right)\|u_a\|^2&\leq \|u_a\|^2-\lambda\int_\Omega fu_a^2dx=\int_\Omega gu_a^4dx-a\|u_a\|^4\leq(\Gamma_0-a)\|u_a\|^4.
\end{align*}
This concludes the proof.
\end{proof}

\begin{Theorem}
Suppose that $a>\Gamma_0.$ Let $\lambda_n\to \lambda_1^+(f)$ and let $u^-_n$ be the solution obtained in Theorem \ref{T43} $(ii)$ for Equation $(K_{a,\lambda_n})$. Then we have
$$u^-_n\to0\quad\quad\text{and}\quad\quad J_{a,\lambda_n}(u^-_n)\to0.$$
\end{Theorem}

\begin{proof}
From the proof of Theorem \ref{T43} $(ii)$, we have
\begin{equation}\label{4.10}
J_{a,\lambda_n}(u^-_n)=\inf_{u\in H^1_0(\Omega)}J_{a,\lambda_n}(u)\leq\inf_{t\in\mathbb{R}}J_{a,\lambda_n}(t\phi_1).
\end{equation}
By a direct calculation, we have
\begin{equation}\label{4.11}
\inf_{t\in\mathbb{R}}J_{a,\lambda_n}(t\phi_1)=\frac{-(\lambda_n-\lambda_1(f))^2}{4\left(a\lambda_1(f)^2-\int_\Omega g\phi_1^4dx\right)} \to 0\quad\text{as $n\to\infty$}.
\end{equation}
Combining $\eqref{4.7}, \eqref{4.10}$ and $\eqref{4.11}$ gives $\lim_{n\to\infty}\|u^-_n\|=0$. It follows from \eqref{4.7} and $\lim_{n\to\infty}\|u^-_n\|=0$
that $\lim_{n\to\infty}J_{a,\lambda_n}(u^-_n)=0.$
\end{proof}

\section{The Proofs of Theorems \ref{T:p<4+}-\ref{T:p<4-2}}

In this final section, we present the proofs of Theorems \ref{T:p<4+}-\ref{T:p<4-2} and we start by showing the boundedness of Palais-Smale sequence using the fact that the functional $J_{a,\lambda}$ is coercive and bounded below on $H^1_0(\Omega)$ for $p<4$. The asymptotic behaviour of solutions is briefly discussed after the proofs of the main theorems.

\begin{Lemma}\label{BD2-4}
Suppose that $N\geq1$, $2<p<\min\{4,2^*\}$ and conditions $(D1)-(D2)$ hold. Then the functional $J_{a,\lambda}$ is coercive and bounded below on $H^1_0(\Omega)$ for every $a, \lambda>0$.
\end{Lemma}

\begin{proof}
For any $a, \lambda>0$, we have
$$J_{a,\lambda}(u)\geq \frac a4\|u\|^4+\frac12\left(1-\frac{\lambda}{\lambda_1(f)}\right)\|u\|^2-\frac{\|g\|_\infty}{pS_p^p}\|u\|^p,$$
and Young's inequality gives
\begin{equation*}
\frac{\|g\|_\infty}{pS_p^p}\|u\|^p\leq\frac {a}{16}\|u\|^4+\frac{4-p}{4p}\left(\frac{\|g\|_\infty}{S_p^p}\right)^{\frac{4}{4-p}}\left(\frac 4a\right)^{\frac{p}{4-p}}.
\end{equation*}
Thus, for $0<\lambda\leq\lambda_1(f)$, we have
\begin{equation}\label{2}
J_{a,\lambda}(u)\geq\frac{3}{16}a\|u\|^4+ \frac12\left(1-\frac{\lambda}{\lambda_1(f)}\right)\|u\|^2-\frac{4-p}{4p}\left(\frac{\|g\|_\infty}{S_p^p}\right)^{\frac{4}{4-p}}\left(\frac 4a\right)^{\frac{p}{4-p}}.
\end{equation}
Using Young's inequality when $\lambda>\lambda_1(f)$, we obtain
\begin{equation*}
\frac12\left|1-\frac{\lambda}{\lambda_1(f)}\right|\|u\|^2\leq\frac {a}{16}\|u\|^4+\frac1a\left|1-\frac{\lambda}{\lambda_1(f)}\right|,
\end{equation*}
and thus
\begin{equation}\label{3}
J_{a,\lambda}(u)\geq\frac a8\|u\|^4- \frac1a\left|1-\frac{\lambda}{\lambda_1(f)}\right|-\frac{4-p}{4p}\left(\frac{\|g\|_\infty}{S_p^p}\right)^{\frac{4}{4-p}}\left(\frac 4a\right)^{\frac{p}{4-p}}.
\end{equation}
Hence $J_{a,\lambda}$ is coercive and bounded below on $H^1_0(\Omega)$ for every $a, \lambda>0$.
\end{proof}

{\bf Now we proceed to the proof of Theorem \ref{T:p<4+}:} $(i)$ By Lemma \ref{MP:p<4+}, for each $a>0$, the functional $J_{a,\lambda}$ has the mountain pass geometry for every $\max\{0, \lambda_a^+\}<\lambda<\lambda_1(f)$. Then using Lemma \ref{BD2-4}, we can deduce that the $(PS)_\alpha$-sequence for $J_{a,\lambda}$ is bounded for any $\alpha\in\mathbb{R}$. Thus, by Lemma \ref{LPS}, there exists $u^+\in H^1_0(\Omega)$ such that $J_{a, \lambda}(u^+)>0$ and $J'_{a,\lambda}(u^+)=0$. Since $J_{a,\lambda}(u^+)=J_{a,\lambda}(|u^+|)$, we assume without loss of generality that $u^+>0$ in $\Omega$.

Next, we consider the infimum of $J_{a,\lambda}$ on the set $\{u\in H^1_0(\Omega) : \|u\|\geq\rho_{a,\lambda}^+\}$ with $\rho_{a,\lambda}^+$ as given in Lemma \ref{MP:p<4+}.
Set
$$\beta_1:=\inf_{\|u\|\geq\rho_{a,\lambda}^+}J_{a,\lambda}(u).$$
It follows from Lemmas \ref{MP:p<4+} and \ref{BD2-4} that $-\infty<\beta_1<0$. By the Ekeland variational principle \cite{E}, there exists a $(PS)_{\beta_1}$-sequence $\{u_n\}$ which is bounded according to Lemma \ref{BD2-4}. Then by Lemma \ref{LPS}, there exists $u^-\in H^1_0(\Omega)$ such that $J_{a,\lambda}(u^-)=\beta_1$ and $J'_{a,\lambda}(u^-)=0$. Since $J_{a,\lambda}(u^-)=J_{a,\lambda}(|u^-|)$, we may take $u^->0$ in $\Omega$.

$(ii)$ We consider the infimum of $J_{a,\lambda}$ on $H^1_0(\Omega)$. Set
$$\beta_2:=\inf_{u\in H^1_0(\Omega)}J_{a,\lambda}(u),$$
and $\beta_2$ is finite by Lemma \ref{BD2-4}.
For any $t>0$,
\begin{equation*}
J_{a,\lambda}(t\phi_1)= -\frac {\lambda-\lambda_1(f)}{2}t^2-\frac{\int_\Omega g\phi_1^pdx}{p}t^p+\frac{a\lambda_1(f)^2}{4}t^4.
\end{equation*}
Thus, there exists $t_0>0$ such that $J_{a,\lambda}(t_0\phi_1)<0$ and so $\beta_2<0$. By the Ekeland variational principle \cite{E}, there exists a $(PS)_{\beta_2}$-sequence $\{u_n\}\subset H^1_0(\Omega)$ and the sequence is bounded as a result of Lemma \ref{BD2-4}. Then by Lemma \ref{LPS}, there exists $u^-\in H^1_0(\Omega)$ such that $J_{a,\lambda}(u^-)=\beta_2$ and $J'_{a,\lambda}(u^-)=0$. Since $J_{a,\lambda}(u^-)=J_{a,\lambda}(|u^-|)$, we may assume that $u^->0$ in $\Omega$.

$(iii)$ Note that if $u_0\neq0$ is a solution, then
$$\langle J_{a,\lambda}'(u_0), u_0\rangle=a\|u_0\|^4+\|u_0\|^2-\lambda\int_\Omega fu_0^2dx-\int_\Omega g|u_0|^pdx=0.$$
Thus, we consider the following:
$$I_p(u):=a\|u\|^4+\|u\|^2-\lambda\int_\Omega fu^2dx-\int_\Omega g|u|^pdx\quad\text{for $u\in H^1_0(\Omega)\setminus\{0\}$.}$$
For $\lambda>0$, we may deduce that
$$I_p(u)\geq a\|u\|^4+\left(1-\frac{\lambda}{\lambda_1(f)}\right)\|u\|^2-\frac{\|g\|_\infty}{S_p^p}\|u\|^{p}.$$
We consider the function $H : \mathbb{R}^+\to\mathbb{R}$ which is defined by
$$H(x)=ax^2+\left(1-\frac{\lambda}{\lambda_1(f)}\right)-\frac{\|g\|_\infty}{S_p^p}x^{p-2}, \quad x>0.$$
It is easy to obtain the absolute minimum value
$$H(x_0)=1-\frac{\lambda}{\lambda_1(f)}-(4-p)\left(\frac{\|g\|_\infty}{2S_p^p}\right)^{\frac{2}{4-p}}\left(\frac{p-2}{a}\right)^{\frac{p-2}{4-p}}$$
at $x_0=\left(\frac{(p-2)\|g\|_\infty}{2aS_p^p}\right)^{\frac{1}{4-p}}$. Then we have $H(x_0)>0$ as $\lambda<\Lambda_a^+$,
which implies that $H(x)>0$ for all $x>0$ as $\lambda<\Lambda_a^+$. Hence we conclude that
$$I_p(u)\geq\|u\|^2\left(a\|u\|^2+\left(1-\frac{\lambda}{\lambda_1(f)}\right)-\frac{\|g\|_\infty}{S_p^p}\|u\|^{p-2}\right)>0\quad\text{for all $u\in H^1_0(\Omega)\setminus\{0\}$},$$
whenever $0<\lambda<\Lambda_a^+$. This completes the proof.

Next, we investigate the asymptotic behaviour of solutions.

\begin{Theorem}
For each $a>0$ and $\lambda=\lambda_1(f)$, let $u_a$ be the solution obtained in Theorem \ref{T:p<4+} $(ii)$. Then $u_a\to0$ as $a\to\infty$.
\end{Theorem}

\begin{proof}
This result is readily obtained using \eqref{2} with $J_{a,\lambda_1(f)}(u_a)<0$.
\end{proof}


{\bf Here we present the proof of Theorem \ref{T:p<4-1}:} $(i)$ By Lemma \ref{MP:p<4-1}, for each $0<a<a_0(p)$, there exists $\overline{\delta}_a>0$ such that the functional $J_{a,\lambda}$ has the mountain pass geometry for every $0<\lambda<\lambda_1(f)+\overline{\delta}_a$. Then by Lemma \ref{BD2-4}, we can deduce that each $(PS)_\alpha$-sequence for $J_{a,\lambda}$ is bounded. Thus, by Lemma \ref{LPS}, there exists $u^+\in H^1_0(\Omega)$ such that $J_{a, \lambda}(u^+)>0$ and $J'_{a,\lambda}(u^+)=0$. Since $J_{a,\lambda}(u^+)=J_{a,\lambda}(|u^+|)$, we may have $u^+>0$ in $\Omega$.

Next, we consider the infimum of $J_{a,\lambda}$ on the set $\{u\in H^1_0(\Omega) : \|u\|\geq\overline{\rho}_{a,\lambda}\}$ with $\overline\rho_{a,\lambda}$ as in Lemma \ref{MP:p<4-1}.
Set
$$\beta_1:=\inf_{\|u\|\geq\overline{\rho}_{a,\lambda}}J_{a,\lambda}(u).$$
It follows from Lemmas \ref{BD2-4} and \ref{MP:p<4-1} that $-\infty<\beta_1<0$. By the Ekeland variational principle \cite{E}, there exists a $(PS)_{\beta_1}$-sequence $\{u_n\}$ and the sequence is bounded by Lemma \ref{BD2-4}. Using Lemma \ref{LPS}, there exists $u_1^-\in H^1_0(\Omega)$ such that $J_{a,\lambda}(u_1^-)=\beta_1$ and $J'_{a,\lambda}(u_1^-)=0$. Since $J_{a,\lambda}(u_1^-)=J_{a,\lambda}(|u_1^-|)$, we assume without loss of generality that $u_1^->0$ in $\Omega$. Here we complete the proof of ($i$-1) and the partial proof of $(i$-2).

To complete the proof for the remaining part of $(i$-2), for each $\lambda_1(f)<\lambda<\lambda_1(f)+\overline{\delta}_a$, we consider the infimum of $J_{a,\lambda}$ on the closed ball $\overline{B}_{\overline{\rho}_{a,\lambda}}=\{u\in H^1_0(\Omega) : \|u\|\leq\overline{\rho}_{a,\lambda}\}$.
Set
$$\beta_2:=\inf_{\|u\|\leq\overline{\rho}_{a,\lambda}}J_{a,\lambda}(u).$$
For any $t>0$,
$$ J_{a,\lambda}(t\phi_1)=-\frac{\lambda-\lambda_1(f)}{2}t^2+\frac{|\int_\Omega g\phi_1^pdx|}{p}t^p+\frac {a\lambda_1(f)^2}{4}t^4.$$
Clearly, there exists $t_0>0$ such that $\|t_0\phi_1\|<\overline{\rho}_{a,\lambda}$ and $J_{a,\lambda}(t_0\phi_1)<0$
giving $-\infty<\beta_2<0$. By the Ekeland variational principle \cite{E}, there exists a $(PS)_{\beta_2}$-sequence $\{u_n\}\subset B_{\rho_{a,\lambda}}$. Then using Lemma \ref{LPS}, there exists $u_2^-\in H^1_0(\Omega)$ such that $J_{a,\lambda}(u_2^-)=\beta_2$ and $J'_{a,\lambda}(u_2^-)=0$. Since $J_{a,\lambda}(u_2^-)=J_{a,\lambda}(|u_2^-|)$, we may have $u_2^->0$ in $\Omega$.

$(ii)$ This is proved using the same argument of Theorem \ref{T:p<4+} $ (ii)$ and is not repeated here.

$(iii)$ Using the same argument in the proof of Theorem \ref{T:p<4+} $ (iii)$, we consider the following:
 $$I_p(u):=a\|u\|^4+\|u\|^2-\lambda\int_\Omega fu^2dx-\int_\Omega g|u|^pdx\quad\text{for $u\in H^1_0(\Omega)\setminus\{0\}$.}$$
We decompose each $u$ as $u=t\phi_1+w$, where $t\in\mathbb{R}$, $w\in H^1_0(\Omega)$ and $\int_\Omega \nabla w \nabla\phi_1dx=0$. Then for $0<\lambda\leq\lambda_1(f)$, we have
\begin{equation}\label{1.6.1}
\|u\|^2-\lambda\int_\Omega fu^2dx\geq(\lambda_1(f)-\lambda)t^2+\left(1-\frac{\lambda}{\lambda_2(f)}\right)\|w\|^2\geq \left(1-\frac{\lambda}{\lambda_2(f)}\right)\|w\|^2.
\end{equation}
Repeating the same process in $\eqref{g1}$ and $\eqref{g2}$ gives
\begin{align}\label{1.6.2}
-\int_\Omega g|u|^pdx&=-\int_\Omega g|t\phi_1|^pdx+\left(\int_\Omega g|t\phi_1|^pdx-\int_\Omega g|t\phi_1+w|^pdx\right)\notag\\
\ &\geq \left|\int_\Omega g\phi_1^pdx\right||t|^p-2^{p-2}(p-1)\|g\|_\infty\|\phi_1\|_p^p B^{\frac{p}{p-1}}|t|^p-\frac{2^{p-2}\|g\|_\infty}{S_p^p}\left( B^{-p}+p\right)\|w\|^p\notag\\
\ &> -\frac{2^{p-1}\|g\|_\infty}{S_p^p B^p}\|w\|^p,
\end{align}
where $ B=\left(\frac{|\int_\Omega g\phi_1^pdx|}{2^{p-1}(p-1)\|g\|_\infty\|\phi_1\|_p^p}\right)^{\frac{p-1}{p}}.$
Using Young's inequality then gives
\begin{align}\label{1.6.3}
\ &\ \frac{2^{p-1}\|g\|_\infty}{S_p^p B^p}\|w\|^p\notag\\
\leq&\ \left(1-\frac{\lambda}{\lambda_2(f)}\right)\|w\|^2+(4p-8)(4-p)^{\frac{4-p}{p-2}}\left(\frac{\|g\|_\infty}{S_p^p B^p}\right)^{\frac{2}{p-2}}\left(1-\frac{\lambda}{\lambda_2(f)}\right)^{\frac{p-4}{p-2}}\|w\|^4.
\end{align}
It follows from $\eqref{1.6.1}-\eqref{1.6.3}$ that
\begin{align*}
I_p(u)&> a\|u\|^4-(4p-8)(4-p)^{\frac{4-p}{p-2}}\left(\frac{\|g\|_\infty}{S_p^p B^p}\right)^{\frac{2}{p-2}}\left(1-\frac{\lambda}{\lambda_2(f)}\right)^{\frac{p-4}{p-2}}\|w\|^4\\
\ &> \|u\|^4\left[a-(4p-8)(4-p)^{\frac{4-p}{p-2}}\left(\frac{\|g\|_\infty}{S_p^p B^p}\right)^{\frac{2}{p-2}}\left(1-\frac{\lambda}{\lambda_2(f)}\right)^{\frac{p-4}{p-2}}\right]\\
\ &> 0\quad \text{for $0<\lambda\leq\min\{\lambda_1(f), \Lambda_a^-\}$}.
\end{align*}
We completes this proof.

{\bf The proof of Theorem \ref{T:p<4-2}:} Using Lemma \ref{MP:p<4-2}, this proof is essential same as that of Theorem \ref{T:p<4-1} $ (i)$ and is not repeated here.


$${\bf Acknowledgement}$$
K.H. Wang and T.F. Wu was supported in part by the Ministry of Science and Technology, Taiwan(Grant No. 108-2115-M-390-007-MY2 and Grant No. 108-2811-M-390-500).

\end{document}